\documentclass[10pt,a4paper,final]{article}
\usepackage[utf8]{inputenc}
\usepackage[lmargin=2.5cm,tmargin=2cm,bmargin=2cm,rmargin=2.5cm]{geometry}
\usepackage{amsfonts}
\usepackage{amssymb}
\usepackage{fix-cm}
\usepackage{amscd,amssymb,stmaryrd}
\usepackage{amsmath}
\usepackage{graphicx}
\usepackage{subfigure}
\usepackage{fancybox}
\usepackage{amscd,amstext}
\usepackage{hyperref}
\usepackage{mathabx}
\usepackage{color}
\usepackage{epstopdf}
\usepackage{caption}
\usepackage{multicol}
\usepackage{cite}
\usepackage{amsthm}
\usepackage{mathtools}
\usepackage{bigints}
\usepackage{amsrefs}
\usepackage{showlabels} 
\usepackage{enumitem}
\usepackage{mathrsfs}
\usepackage{tikz}
\usepackage{empheq}
\usepackage{framed}
\usepackage{ulem}
\usepackage{nameref}
\usepackage{esint}

\numberwithin{equation}{section}

\newcommand{\ds}{\displaystyle}

\newtheorem{theorem}{Theorem}[section]
\newtheorem{lemma}[theorem]{Lemma}

\newtheorem{definition}[theorem]{Definition}
\newtheorem{prop}[theorem]{Proposition}

\newtheorem*{theorem*}{Theorem}
\newtheorem*{lemma*}{Lemma}
\newtheorem*{conj*}{Conjecture}
\newtheorem*{corollary*}{Corollary}
\newtheorem*{proposition*}{Proposition}
\newtheorem{remark}[theorem]{Remark}
\newcommand{\rom}[1]{\uppercase\expandafter{\romannumeral #1\relax}}

\newcommand{\lan}{\langle}
\newcommand{\ran}{\rangle}

\newcommand{\Z}{\mathbb{Z}}
\newcommand{\R}{\mathbb{R}}

\newcommand{\bT}{\mathbb{T}}

\newcommand{\N}{\mathbb{N}}

\newcommand{\cH}{\mathcal{H}}

\newcommand{\Ga}{\alpha}
\newcommand{\Gb}{\beta}
\newcommand{\Gd}{\delta}
\newcommand{\Ge}{\varepsilon}
\newcommand{\Gg}{\gamma}

\newcommand{\Go}{\omega}

\newcommand{\Gth}{\theta}
\newcommand{\GG}{\Gamma}

\newcommand{\GD}{\Delta}

\DeclareMathOperator*{\esssup}{ess\,sup}

\makeatletter
\newcommand{\labitem}[2]{%
\def\@itemlabel{\textbf{#1}}
\item
\def\@currentlabel{#1}\label{#2}}
\makeatother

\begin{document}

\begin{center}
	{\Large  Periodic solution and its stability of a damped BBM equation posed on $\mathbb{T}$}\\\vspace{0.25in} 
    Chun-Ho Lau \footnote[1]{Department of Mathematics, National Taiwan Normal University, Taipei City, Taiwan 106308. Email: lauchho@ntnu.edu.tw}
	Taige Wang \footnote[2]{Department of Mathematical Sciences, University of Cincinnati, Cincinnati, OH 45221, USA. Email: taige.wang@uc.edu}\ \ \ \
	   \vspace{0.06in}
\end{center}

\begin{abstract}
	In this paper, we are concerned with the existence and stability of temporal periodic solutions to a class of the Benjamin-Bona-Mahony (BBM) equation with damping in the torus $\mathbb{T}$, whose interior is applied with periodic force $f(x, t)$ with temporal period $\theta$. These types of solutions are established in $H^{\ell}, \ell\ge 0$, and it is noteworthy to see that the I-energy method is invoked when low regularity of $\ell\in[0, 1)$ is pursued.
	
\end{abstract}
\section{Introduction}

 In this manuscript, we pursue the existence and stability of solutions of the damped BBM-Burgers equation prescribed with periodic boundary conditions, equivalent to being posed on the torus $\mathbb{T}:=\R/\Z$. We consider the following forced small-amplitude long-wave equation:

\begin{equation} \label{main1}
\begin{cases}
    u_t+u_x-u_{xxt}+uu_x+\beta u-\gamma u_{xx}=f(x,t), \quad\quad (x, t)\in \bT\times [0, \infty),&\\
    u(x,0)=\phi(x), &
\end{cases}
\end{equation}

\noindent where $f$ is external force applied on $\mathbb{T}$. This paper pursues the results of the above equation which are analogous to periodic solutions when imposing a temporally periodic $f$ and low-regularity function spaces on KdV's by Usman, Yang, Zhang \cites{UsmanZ2, Yang, Yang2}. \\

The equation (\ref{main1}), known as the Benjamin-Bona-Mahony (BBM) equation introduced in \cite{BBM}, similar to its alternative counterpart, the Korteweg-de Vries (KdV) equation, serves as a fundamental model for describing the unidirectional propagation of long water waves in a channel. At the same time, KdV attracts more attention with fruitful results. To the authors' knowledge, the KdV equation on a periodic spatial domain is well-known for its time-reversibility and lack of inherent smoothing; research has shown that the introduction of specific dissipative mechanisms can fundamentally alter these dynamics. For instance, Russell and Zhang (see e.g. \cite{Russell}) demonstrated this for unforced KdV on the periodic domain. A point dissipation feedback condition — specifically a dissipative boundary requirement on the first derivative — breaks the system's conservation laws to induce both Kato-type smoothing and exponential decay of solutions; Later in \cite{RZ} they studied the stabilization of the forced KdV equation on the periodic domain by introducing the feedback control on the right side of the equation. One can also consider damping added to the entire domain, yielding long-time behavior for dispersive waves, as introduced in \cite{BSZ2} for KdV on the half-line. Adding damping such as $\alpha u$ and $-u_{xx}$ brings different stabilization of KdV on the whole line. We refer readers to \cite{BW}. \\

Since time-periodic solutions' stability needs exponential damping (shown in \cite{UsmanZ2, WZ} on bounded domains) for forced oscillations, we extend these concepts to the BBM equation on a torus, particularly under the influence of forced oscillations, which allows for an exploration of how external periodic forcing interacts with the equation's inherent dispersive structure to reach steady-state configurations or forced periodic regimes. On the same equation posed on the whole line, \cite{Amick} has pointed out that a similar stabilization mechanism leading to exponential decay can work for BBM equations. In this paper, we consider adding $\Gb u - \Gg u_{xx}$ to obtain the stabilization regime. For this model, several works have studied the global well-posedness (GWP) of the BBM equation under similar conditions. In particular, we highlight \cite{HChen}, which considers spatially periodic initial data without damping.

Beyond establishing time-periodic solutions, the well-posedness theory for related dispersive models plays a crucial role in deriving further results. As pointed out by Bona and Tzvetkov \cite{Tzvet}, while the BBM equation is globally well-posed for $\ell \ge 0$, the authors establish its ill-posedness for $\ell < 0$ by proving that the solution map cannot be $C^2$-smooth. This implies that the equation cannot be solved using traditional iteration schemes in these negative-indexed Sobolev spaces, indicating that the well-posedness result at $\ell = 0$ is sharp. On the contrary, the KdV, local well-posedness (LWP) on $\mathbb{R}$ is known to hold in $H^\ell, \ell>-{3\over 4}$ (see e.g. \cite{Bourgain, Kenig}) and improved global well-posedness is known to hold also in $H^\ell$ with $\ell > -{3\over 4}$ (see e.g. \cite{Tao1} by Colliander et al); in addition, $\ell$ is $-{1\over 2}$ when considering domain being torus $\mathbb{T}$. This value was ever considered as a sharp limit, as it has been shown that a solution cannot be obtained through the standard fixed point theorem when $\ell$ falls below this point. The authors introduced a smoothing operator $I_N$, which maps rougher functions in $H^\ell$ to the high-regularity space $L^2$. It can be shown that while the $L^2$ norm of $Iu$ is not perfectly conserved, it is ``almost conserved"—it grows slowly enough that you can iterate the local existence result to cover any time interval $[0, T]$. This is known as the I-energy method working on those ``almost conserved quantities". Colliander et al also used this seminal technique along with multilinear estimates for generalized KdV on $\mathbb{T}$ in \cite{Tao2}. \cite{Tzvet} and \cite{MWang2} applied the same method on the BBM equation to reach GWP and attractor results. Current authors also considered the same equation and its time-periodic solution on the whole line using a similar technique (see e.g. \cite{LW2}). In this manuscript, we also extend this framework to the BBM equation and obtain a corresponding result in spaces below $H^1$. 

Before narrating our main theorems, we need to quantify two types of stability that the solution may possess to align with long-time behavior as follows:

\begin{definition}\label{de2-2}
    Regarding the dynamics of differential equations, we say a solution $\tilde u(t)$ generated by an initial data $u_0\in X$ is locally stable if $u(t)$ converges to $\tilde u$ in $X$ as $t\rightarrow \infty$, when $u_0$ is sufficiently close to $\tilde u$.
\end{definition}

\begin{definition}\label{de2-3}
    We say $\tilde u(t)$ is globally stable if  $u(t)$ generated by any initial data $u_0\in X$ converges to $\tilde u$ in $X$ regardless of the distance between $u_0$ and $\tilde u$.
\end{definition}

The following are the theorems for $\ell\ge 0$:

\begin{theorem} \label{thm1}
   Suppose $\Gb\in \R$ and $\Gg>\max\{0, -\frac{\Gb}{4\pi^2}\}$. For $T, \tau>0$, if $\lan f(t)\ran=0$ for all $t\geq0$, and  there exists $\Gd>0$ (independent of $\tau$) such that 
    $$\|\phi\|_{H^{\ell}_{per, 0}}+\|f\|_{L^{\infty}([0,\infty);H^{\ell}_{per, 0})}\leq \delta,$$ then there exists a unique $u\in L^{\infty}\cap L^2([\tau,\tau+T];\dot{H}^\ell_{per,0})$ such that $u$ is a mild solution to \eqref{main1} and  
    $$\|u\|_{Y^{\ell}_{\tau,T}}\leq C(\|\phi\|_{H^{\ell}_{per, 0}}+\|f\|_{L^{\infty}([0,\infty);H^{\ell-2}_{per, 0})}). $$
\end{theorem}
Here the space $H^{\ell}_{per, 0}$ denotes the Sobolev space using Bessel potential on $\bT$ with integral 0, $Y^\ell_{\tau,T}$ consists of the same set as $L^{\infty}\cap L^2([\tau,T];\dot{H}^\ell_{per})$, and $\lan y\ran=\int_{\bT}y(x)dx$. The definitions are given in Section \ref{pre}. \\

We remark that $\Gd$ is independent of $T$ and $\tau$ if one considers $f\in L^2([0,\infty);H^{\ell}_{per, 0})$.\\

\begin{theorem} \label{thm2}
     Suppose the assumptions in Theorem \ref{thm1} hold and $f$ is $\Gth$-periodic in time with $\lan f(t)\ran=0$ for $t\in [0,\Gth)$. For $T, \tau>0$, there exists $C>0$ and $\rho>0$ such that 
    $$\|u(\cdot, \cdot+\theta)-u(\cdot, \cdot)\|_{Y^{\ell}_{\tau,T}}\leq C e^{-\rho \tau}$$
    whenever $\|\phi\|_{H^{\ell}_{per, 0}}+\|f\|_{L^{\infty}([0,\infty);H^{\ell}_{per, 0})}\leq \delta$.
\end{theorem}
\begin{theorem} \label{thm3}
   Suppose the assumptions in Theorem \ref{thm2} hold. Then \eqref{main1} has a locally stable and unique periodic solution in $H^{\ell}_{per, 0}$. 
\end{theorem}
\begin{theorem}  \label{thm4}
   Suppose $\Gb\in \R$ and $\Gg>\max\{0, -\frac{\Gb}{4\pi^2}\}$, and  $f$ is $\Gth$-periodic in time with $\lan f(t)\ran=0$ for $t\in [0,\Gth)$. Then, there exists $\Gd'>0$ such that \eqref{main1} has essentially globally stable periodic solution whenever $\displaystyle\sup_{t\geq 0}\|f(t)\|_{H^{\ell-2}_{per, 0}}\leq \Gd'$.
\end{theorem}

\begin{remark}
Theorem \ref{thm2} gives asymptotic periodicity, given an imposed temporally periodic $f$, which is related to exponential stability of a version of the linearized damped BBM equation. This property is observed in many infinite-dimensional dynamical systems and constitutes a key estimate for establishing exact periodicity and uniqueness.
\end{remark}

\begin{remark}
Interplay of $\Gb, \Gg$ presented in Theorem \ref{thm1} guarantees the exponential stability of time-periodic solutions of the original damped BBM equation. In conclusion, the diffusion term $-\Gg u_{xx}$ strongly contributes to the exponential stability more than $\Gb u$ on the left-hand side. If $\Gb>0$, $\Gb u$ contributes to stability, while $\Gb<0$ with not sufficiently large amplitude, stability is still achievable thanks to $\Gg>0$. 
\end{remark}


Since slightly different methodologies are employed for the cases of $\ell\gtrless1$, we organize the paper as follows: Section \ref{pre} is dedicated to introducing function spaces and notations including their norms, and a key lemma treating the nonlinear convection term; Section \ref{Sec3} includes a series of lemmas and propositions to lead to proof of Theorem \ref{thm1} for $H^{\ell}, \ell\ge 1$; Section \ref{Sec4} gives the proof of Theorem \ref{thm2}, \ref{thm3}, \ref{thm4}; as the previous sections are all about energy spaces of $\ell\ge 1$, Section \ref{Sec5} is dedicated to giving a key sketch of proving similar results but in low-regularity spaces of $\ell\in[0, 1)$ using the I-method. 

\section{Preliminary} \label{pre}

In this section, first we introduce a series of function spaces compatible with the periodic boundary nature: 

\begin{itemize}
    \item Square integrable space $$L^2_{per}:=\ds \bigg\{f: f(x+1)=f(x), \int_0^1|f(x)|^2dx<\infty\bigg\}, $$
    whose norm is simply denoted by $\|f\|.$
    \item Let $\ell\in \R$. The $L^2$-based Sobolev space using Bessel potential $$H^{\ell}_{per}=\bigg\{f\in \mathcal{D}'_{per}: \ds \sum_{k\in\Z} (1+(2\pi k)^2)^{\ell}|\widehat{f}(k)|^2<\infty\bigg\},$$ where the Fourier coefficients of $f$ are defined to be 
    $\ds \widehat{f}(k):=\int_0^1 e^{-2\pi i k x}f(x)dx=\int_{\bT} e^{-2\pi i kx}f(x)dx$, and the norm on $H^\ell_{per}$ is defined to be 
    $$\|f\|_{H^{\ell}_{per}}=\bigg( \sum_{k\in\Z} (1+(2\pi k)^2)^{\ell}|\widehat{f}(k)|^2\bigg)^{\frac{1}{2}}=\|(I-\Delta)^{\frac{\ell}{2}}f\|_{L^2_{per}},$$
    where $\Delta = \partial^2_{x}$ denotes the Laplacian on $\bT$. The space of all $H^{\ell}_{per}$ elements with integral 0 will be denoted by $H^{\ell}_{per, 0}$. The norm on $H^{\ell}_{per, 0}$ is the same as the norm on $H^{\ell}_{per}$. By abuse of notation, throughout the remainder of this paper we write $\|f\|_{H^{\ell}_{per, 0}}=\|f\|_{H^{\ell}_{per}}$ if $f\in H^{\ell}_{per, 0}$.
    \item The space $L^p([0,T];X)$ is defined to be all elements $f$ such that $f(t)\in X$ for all $t\in [0,T]$ and 
    $$\int_0^T\|f(t)\|_{X}^pdt<\infty,$$
    where $X$ is a function space and $p<\infty$. For $p=\infty$, instead of $L^p$ integral, we assume $f$ satisfies 
    $$\esssup_{t\in [0,T]} \|f(t)\|_X <\infty.$$
    In this context, given $f(x, t)\in X$ for each $t\ge 0$, $\|f(t)\|_X$ means $\|f(\cdot, t)\|_X$ which follows the fashion of ODE following time evolution.
    \item Function space $Y^\ell_{\tau, T}$ related to $H_{per}^\ell$ is endowed with the norm
 \begin{equation*}
\|u\|_{Y^\ell_{\tau, T}}=\sup_{\tau\le s\le \tau +T}\|u(s)\|_{H_{per}^\ell}+\left(\int_{\tau}^{\tau+T}\|u(s)\|^2_{H_{per}^\ell}ds\right)^{1\over 2}.
 \end{equation*} \\
\end{itemize}

Throughout this paper, for $f\in L^{\infty}([0,T]; L^1_{per})$, we write
$$\langle f(t)\rangle:=\int_{\bT} f(t,x)dx.$$

We frequently use the following estimates without explicitly mentioning them
$$\|u\|_{H^{\ell}_{per}}^2=\|u\|_{H^{\ell-1}_{per}}^2+\|u_x\|_{H^{\ell-1}_{per}}^2$$
for any $\ell\in \R$, 
and  
$$\|uv\|_{H^{\ell}_{per}}\leq C\|u\|_{H^{\ell}_{per}}\|v\|_{H^{\ell}_{per}}$$
whenever $\ell>\frac{1}{2}$. The second inequality shows that $H^{\ell}_{per}$ forms an algebra whenever $\ell>\frac{1}{2}$ due to the embedding of Sobolev spaces defined in 1D domains. 


\begin{lemma} \label{bilinear}
    Let $\ell\geq 1$ and $0\leq \ell'\leq \ell-1$. Suppose $v,w\in L^{\infty}([0,T]; H^\ell_{per})\cap L^2([0,T]; H^\ell_{per})$. Then, for all $s\in [0,T]$, and $t\in [0,T-s]$, we have 
     \begin{align*}
        \int_s^{s+t}\|(vw)_x\|_{H^{\ell'}_{per}}^2dt'\leq C'' \|v\|_{Y^{\ell}_{s,t}}^2\|w\|_{Y^{\ell}_{s,t}}^2,
    \end{align*}
    where constant $C^{''}$ is independent of $s$, $t$, and $T$.
\end{lemma}
The general version is proved in Remark \ref{multi}. For an alternative proof of this inequality, one may refer to \cite{LW}*{Lemma 3.11}.

\section{Linear BBM with damping of the form $\Gb u-\Gg u_{xx}$} \label{Sec3}
In this section, we shall focus on:
\begin{equation} \label{sub1}
\begin{cases}
    v_t+v_x-v_{xxt}+\beta v-\gamma v_{xx}=f(x,t), \quad\quad (x, t)\in \bT\times [0, \infty)&\\
    v(x,0)=\phi(x), &\\
    \langle v(t) \rangle=0,  \quad\quad t\geq 0.&
\end{cases}
\end{equation}


Define the operator $Av:=(I-\Delta)^{-1}(-v_x-\Gb v+\Gg v_{xx})$ so that the equation reads compactly

$$v_t = Av + (I-\triangle)^{-1}f.$$

We would like to pursue when $A$ generates a semigroup $e^{At}$ on $H^{\ell}_{per, 0}$ with $\|e^{At}\|_{L(H^{\ell}_{per})}\leq e^{-\omega t}$ for some $\omega>0$, which is the following proposition:

\begin{prop}\label{prop3-1}
    Suppose $\ell\ge0$, $\Gb\in \R$, and $\Gg>\min\{0, \frac{-\Gb}{4\pi^2}\}$. The operator $A$ has the following properties:
    \begin{enumerate}
        \item $A$ is a closed operator on $H^{\ell}_{per, 0}$;
        \item $A$ generates a $C_0$-semigroup on $H^{\ell}_{per, 0}$, and there is $\Go>0$ such that $\|e^{At}\|_{L(H^{\ell}_{per})}\leq e^{-\Go t}$ for all $t\geq 0$.
    \end{enumerate}
\end{prop}

\begin{proof}
We first consider $Au=g$, where $g\in C^{\infty}_c(0,1)\cap L^2_{per}$. Thanks tothe  mean-zero condition, the Fourier series of the solution reads

$$u(x, t)\sim \sum_{k\in \Z\setminus \{0\}} -\frac{1+4\pi^2k^2}{\Gg(4\pi^2k^2)+i(2\pi k)+\Gb}\widehat{g}(k)e^{2\pi ikx}.$$

We can see that the series converges in $\ell^2$ as 
$$\bigg|-\frac{1+4\pi^2k^2}{\Gg(4\pi^2k^2)+i(2\pi k)+\Gb}\bigg|\leq \frac{1}{|\Gg|}.$$

Indeed, 
\begin{align}\label{3-2}
   \|u\|_{H^{\ell}_{per}}^2=\sum_{k\in \Z\setminus \{0\}} (1+4\pi^2k^2)^{\ell}\bigg|-\frac{1+4\pi^2k^2}{\Gg(4\pi^2k^2)+i(2\pi k)+\Gb}\bigg|^2|\widehat{g}(k)|^2 &\leq      \sum_{k\in \Z\setminus \{0\}} \frac{1}{|\Gg|^2} (1+(2\pi k)^2)^{\ell }|\widehat{g}(k)|^2 \nonumber\\
   &\leq \frac{1}{|\Gg|^2}\|g\|_{H^{\ell}_{per}}^2.
\end{align}



It is well-known that $A$ is a compact operator on $H^{\ell}_{per}$ as inverse elliptic operator $(I-\triangle)^{-1}$ is compact, and in particular, $A$ is a closed operator. That is the 1st property stated by the Proposition.\\

If $\gamma> 0$, $\Gb\geq 0$ and owing to the commutativity of derivatives including elliptic operators and integration by parts, one can derive

\begin{eqnarray*}
    \langle Au,u\rangle_{H^\ell_{per}}&=& -\langle(I-\triangle)^{{\ell\over2}-1}u_x,(I-\triangle)^{{\ell\over2}}u\rangle-\beta\langle(I-\triangle)^{{\ell\over2}-1}u,(I-\triangle)^{{\ell\over2}}u\rangle-\gamma\langle(I-\triangle)^{{\ell\over2}-1}u_x,(I-\triangle)^{{\ell\over2}}u_x\rangle \\
    &=&-\beta\|u\|_{H^{\ell-1}_{per}}^2-\gamma\|u_x\|_{H^{\ell-1}_{per}}^2\leq-\frac{4\pi^2\gamma}{1+4\pi^2}\|u\|_{H^{\ell}_{per}}^2, 
\end{eqnarray*} where the first inner product term vanishes, by using Poincar\'e inequality for all $\ell\ge 0$. In fact, this version of the Poincaré inequality can be obtained for $H^{\ell}_{per, 0}$:

\begin{align*}
    \|u\|_{H^{\ell}_{per}}^2=\sum_{k\neq 0} (1+4\pi^2 k^2)^{\ell}|\widehat{u}(k)|^2 &\leq \sum_{k\neq 0} \frac{1+4\pi^2k^2}{4\pi^2k^2} (1+4\pi^2 k^2)^{\ell-1}|(2\pi i k)\widehat{u}(k)|^2\leq \frac{1+4\pi^2}{4\pi^2}\|u_x\|_{H^{\ell-1}_{per}}^2.
\end{align*}


If $\beta\leq 0, \gamma>0$, then 
\begin{align*}
    \langle Au,u\rangle_{H^\ell_{per}}&=(-\beta)\|u\|_{H^{\ell-1}_{per}}^2-\gamma\|u_x\|_{H^{\ell-1}_{per}}^2\\
    &=(-\Gb)\|u\|_{H^{\ell-1}_{per}}^2+(-\Gb)\|u_x\|_{H^{\ell-1}_{per}}^2+(\Gb-\Gg)\|u_x\|_{H^{\ell-1}_{per}} \\&\leq \bigg(-\beta+\frac{4\pi^2(\Gb-\Gg)}{1+4\pi^2}\bigg )\|u\|_{H^{\ell}_{per}}^2=\frac{-\Gb-4\pi^2\Gg}{1+4\pi^2}\|u\|_{H^{\ell}_{per}}^2.
\end{align*}
The upper bound has a negative coefficient provided that $\Gg>-\frac{\Gb}{4\pi^2}$. \\

Because of the dissipativity of $A$ in $H_{per}^\ell$ and Estimate (\ref{3-2}), by the renowned Philip-Lumer Theorem (in many literatures such as \cite{LZ}), $A$ generates an analytic $C_0$-semigroup $\{e^{At}\}_{t\ge0}$ with exponential decay and the 2nd property is proved. 
\end{proof}



Proposition \ref{prop3-1} allows us to write the mild solution to \eqref{sub1} as follows:

$$v(t)=(e^{At}(I-\Delta)^{-1}\phi)(x)+\int_0^t(e^{A(t-s)}(I-\Delta)^{-1}f(\cdot,s))(x)ds.$$

Indeed, we have the following estimates:

\begin{lemma}\label{estu}
    Let $\ell\geq 1$. For given $T>0$, there exists $C>0$ such that 
    $$\|v\|^2_{Y^{\ell}_{0,T}}+\|v_t\|^2_{Y^{\ell}_{0,T}}\leq C\bigg(\|\phi\|^2_{H^{\ell}_{per}}+\int_0^T\|f(\cdot,t)\|_{H^{\ell-2}_{per}}^2dt\bigg)$$
\end{lemma}
\begin{proof}
    From the expression of $v$, we have 
    \begin{align*}
        \|v(t)\|_{H^{\ell}_{per}}&\leq \|e^{At}\|_{L(H^{\ell}_{per})}\|\phi\|_{H^{\ell}_{per}}+\int_0^t \|e^{At}\|_{L(H^{\ell}_{per})}\|f(\cdot,s)\|_{H^{\ell-2}_{per}} ds\\
        &\leq e^{-\Go t} \|\phi\|_{H^{\ell}_{per}}+\int_0^te^{-\Go t}\|f(\cdot,s)\|_{H^{\ell-2}_{per}} ds.
    \end{align*}

    To obtain an $L^2$ estimate for $v$, we multiply \eqref{sub1} by $(I-\Delta)^{\ell-1}v$ and integrate over $x\in \bT$, we have 
    \begin{align*}
        \frac{d}{dt}\|v(t)\|^2_{H^{\ell}_{per}}+\int_{\bT}v_x(t)(I-\Delta)^{\ell}v(t)+\Gb\|v(t)\|^2_{H^{\ell-1}_{per}}+\gamma \|v_x(t)\|_{H^{\ell-1}_{per}}^2=\int_\bT f(\cdot,t)(I-\Delta)^{\ell-1}v(t),
    \end{align*}
    and notice that 
    $$\int_\bT v_x(t)(I-\Delta)^{\ell-1}v(t)=\int_\bT [(I-\Delta)^{\frac{\ell-1}{2}}v(t)]_x (I-\Delta)^{\frac{\ell-1}{2}}v(t) =0$$
    thanks to the periodic condition and the commutativity of derivatives. Moreover, on the right-hand side, we have 
    $$ \int_\bT f(x,t)(I-\Delta)^{\ell-1}v(x,t)dx = \int_{\bT}(I-\Delta)^{\frac{\ell}{2}-1}f(x,t) (I-\Delta)^{\frac{\ell}{2}}v(x,t)dx.$$
    We now integrate both sides with respect to $t\in [0,T]$ and get
    \begin{align*}
       &\beta \int_0^T\|v(t)\|^2_{H^{\ell-1}_{per}}dt+ \gamma\int_0^T\|v_x(t)\|_{H^{\ell-1}_{per}}^2dt \\
       &\leq \|\phi\|_{H^{\ell}_{per}}^2+\frac{1}{4\Ge}\int_0^T\|f(\cdot,t)\|_{H^{\ell-2}_{per}}^2dt+\Ge \int_0^T\|v(t)\|^2_{H^{\ell}_{per}}dt.
    \end{align*}
    Note that if both $\gamma, \beta$ are positive, we can take the lower bound as $$C\gamma \int_0^T\|v(t)\|_{H^{\ell}_{per}}^2dt$$ while  if $\gamma> \frac{{1+4\pi^2}}{4\pi^2}(-\Gb)\geq 0$, then we can take the lower bound as 
    $$C(\Gb+\Gg)\int_0^T\|v(t)\|_{H^{\ell}_{per}}^2dt>0.$$

    By choosing $\Ge<C\min\{\Gb+\Gg,\Gg\}$, we infer  
    $$\int_0^T\|v(t)\|_{H^{\ell}_{per}}^2dt\leq C\bigg(\|\phi\|_{H^{\ell}_{per}}^2+\int_0^T\|f(\cdot,t)\|_{H^{\ell-2}_{per}}^2dt\bigg). $$

    The estimate for $v_t$ can be obtained analogously from the estimate above due to the fact 
    $$v_t=(I-\Delta)^{-1}(-v_x-\Gb v +\Gg v_{xx}+f(x,t) )$$
along with $\|v_x\|_{H^{\ell-2}_{per}}, \|v_{xx}\|_{H^{\ell-2}_{per}}, \|v\|_{H^{\ell-2}_{per}}\leq C'\|v\|_{H^{\ell}_{per}}$.
\end{proof}

We now consider the solution $u$ to \eqref{main1} and write it as $u=v+w$, where $v$ is the solution to \eqref{sub1} and $w$ is a solution to 

\begin{equation}\label{eqsub2}
\begin{cases}
    w_t+w_x-w_{xxt}+ww_x+(vw)_x+\Gb w-\Gg w_{xx}=-vv_x,  \quad (x, t)\in \bT\times [0, \infty),&\\
    w(x,0)=0&\\
    \langle w(t)\rangle =0, \quad t\geq 0.&
    \end{cases}
\end{equation}

Note that the solution to \eqref{eqsub2} can be formally written as $$ w(t)=-\int_0^t e^{A(t-s)}(I-\Delta)^{-1}[w(s)w_x(s)+(v(s)w(s))_x+v(s)v_x(s)]ds.$$
\\

We now show the existence and uniqueness, and some estimates of $w$.

\begin{prop} \label{estw}
 Suppose $\ell\geq 1$ and fixed $T>0
 $. Let $\phi\in H^{\ell}_{per}$ and $f\in L^2([0,T];H^{\ell-2}_{per})$. Then, there exists $c>0$ such that if $\|\phi\|_{H^{\ell}_{per}}^2+\|f\|_{L^2([0,T];H^{\ell-2}_{per})}\leq c$, then \eqref{eqsub2} has a unique solution. Moreover, 
 $$\|w\|_{Y^{\ell}_{0,T}}\leq \bigg(\|\phi\|_{H^{\ell}_{per}}^2+\int_0^T\|f(\cdot,t)\|_{H^{\ell-2}_{per}}^2dt\bigg)^{\frac12}.$$
\end{prop}
\begin{proof}
Define $Y^{\ell}_{0,T;M'}:=\{w\in Y^{\ell}_{0,T}: \|w\|_{Y^{\ell}_{0,T}}\leq M'\}$ and the operator 
$$\Gamma(w)(t):=-\int_0^t e^{A(t-s)}(I-\Delta)^{-1}[w(s)w_x(s)+(v(s)w(s))_x+v(s)v_x(s)]ds. $$ It has no harm to assume that 
$$M'=K\bigg(\|\phi\|_{H^{\ell}_{per}}^2+\int_0^T\|f(\cdot,t)\|_{H^{\ell-2}_{per}}^2dt\bigg)^{\frac12}$$
for some $K\geq1$ which will be determined. \\

We will show that $\Gamma: Y^{\ell}_{0,T;M'}\mapsto Y^{\ell}_{0,T;M'}$ and it is a contraction mapping. \\

Let $w\in Y^{\ell}_{0,T;M'}$. By Lemma \ref{bilinear}, 
\begin{align*}
& \|\Gamma(w)(t)\|_{H^{\ell}_{per}}\\
&\leq \int_0^t\|e^{A(t-s)}\|_{H^{\ell}_{per}\to H^{\ell}_{per}}\bigg(\|w(s)w_x(s)\|_{H^{\ell-2}_{per}}+\|(vw)_x(s)\|_{H^{\ell-2}_{per}}+\|v(s)v_x(s)\|_{H^{\ell-2}_{per}}\bigg)ds\\
&\leq C_A\bigg[\bigg(\int_0^t\|[ww]_x(s)\|_{H^{\ell-1}_{per}}^2ds\bigg)^{\frac12}+\bigg(\int_0^t\|[wv]_x(s)\|_{H^{\ell-1}_{per}}^2ds\bigg)^{\frac12} +\bigg(\int_0^t\|[vv]_x(s)\|_{H^{\ell-1}_{per}}^2ds\bigg)^{\frac12}\bigg]\\
&\leq C_A'(\|w\|_{Y^{\ell}_{0,t}}^2+\|w\|_{Y^{\ell}_{0,t}}\|v\|_{Y^{\ell}_{0,t}}+\|v\|_{Y^{\ell}_{0,t}}^2).
\end{align*}
Now using Lemma \ref{estu} we have 
$$\sup_{0\leq t\leq T}\|\Gamma(w)(t)\|_{H^{\ell}_{per}}\leq C_A''\bigg[(M')^2+\frac{(M')^2}{K^2}\bigg].$$

Meanwhile, a similar argument showed that 
\begin{align*}
    \bigg(\int_0^t \|\Gamma(w)(s)\|_{H^{\ell}_{per}}^2ds\bigg)^{\frac12} &\leq C_A'''\bigg[(M')^2+\frac{(M')^2}{K^2}\bigg].
\end{align*}
Thus, we can conclude that 
$$\|\Gamma(w)\|_{Y^{\ell}_{0,T}}\leq (C''+C''')\bigg[(M')^2+\frac{(M')^2}{K^2}\bigg].$$

If we take $(C''+C''')(M')^2(1+K^{-2})\leq M'$, then $\Gamma(w)\in Y^{\ell}_{0,T;M'}$.\\

Next, we also need to estimate the term $\Gamma(w)(s)-\Gamma(w')(s)$ if $w,w'\in Y^{\ell}_{0,T;M'}$. Note that 

\begin{align*}
    &\Gamma(w)(t)-\Gamma(w')(t)\\
    &= -\int_0^t e^{A(t-s)}(I-\Delta)^{-1}[(w(s)[w_x-w'_x](s)+(w-w')(s)w_x'(s))+(v[w-w'])_x(s)]ds.
\end{align*}
Therefore, 
\begin{align*}
    &\sup_{0\leq t\leq T}\|\Gamma(w)(t)-\Gamma(w')(t)\|_{H^\ell_{per}}\\
    &\leq C_A\bigg[\bigg(\int_0^t\|w(s)[w_x-w'_x](s)\|^2_{H^{\ell-2}_{per}}ds\bigg)^{\frac12}+\bigg(\int_0^t\|w'_x(s)[w-w'](s)\|^2_{H^{\ell-2}_{per}}ds\bigg)^{\frac12}\\
    &\quad\quad\quad +\bigg(\int_0^t\|(v[w-w'])_x(s)\|^2_{H^{\ell-2}_{per}}ds\bigg)^{\frac12}\bigg].
\end{align*}
By Lemma \ref{bilinear}, we have 
\begin{align*}
    \int_0^t\|w(s)[w_x-w'_x](s)\|^2_{H^{\ell-2}_{per}}ds&\leq C'' \|w\|^2_{Y^{\ell}_{0,T}}\|w-w'\|^2_{Y^{\ell}_{0,T}};\\
    \int_0^t\|w'_x(s)[w-w'](s)\|^2_{H^{\ell-2}_{per}}ds&\leq C''\|w'\|^2_{Y^{\ell}_{0,T}}\|w-w'\|^2_{Y^{\ell}_{0,T}};\\
    \int_0^t\|(v[w-w'])_x(s)\|^2_{H^{\ell-2}_{per}}ds&\leq C''\|v\|_{Y^{\ell}_{0,T}}^2\|w-w'\|^2_{Y^{\ell}_{0,T}}.
\end{align*}
Thus, 
\begin{align*}
    \sup_{0\leq t\leq T}\|\Gamma(w)(t)-\Gamma(w')(t)\|_{H^{\ell}_{per}}&\leq C_A^{(4)}(\|w\|_{Y^{\ell}_{0,T}}+\|w'\|_{Y^{\ell}_{0,T}}+\|v\|_{Y^{\ell}_{0,T}})\|w-w'\|_{Y^{\ell}_{0,T}}\\
    &\leq C_A^{(4)}(2M'+\frac{M'}{K})\|w-w'\|_{Y^{\ell}_{0,T}}
\end{align*}
and similarly, 
\begin{align*}
    \bigg(\int_0^T\|\Gamma(w)(t)-\Gamma(w')(t)\|_{H^\ell_{per}}^2dt\bigg)^\frac{1}{2}&\leq C_A^{(5)}(\|w\|_{Y^{\ell}_{0,T}}+\|w'\|_{Y^{\ell}_{0,T}}+\|v\|_{Y^{\ell}_{0,T}})\|w-w'\|_{Y^{\ell}_{0,T}}\\
    &\leq C_A^{(5)} (2M'+\frac{M'}{K})\|w-w'\|_{Y^{\ell}_{0,T}}.
\end{align*}
Thus, setting $(C^{(4)}_A+C^{(5)}_A)(2M'+M'/K)<1$ (\textit{i.e.} $M'<\frac{K}{(2K+1)(C^{(4)}_A+C^{(5)}_A)}$) and choosing $K=1$, we can apply Banach Fixed Point Theorem and conclude that $\GG(w)=w$ for some $w\in Y^{\ell}_{0,T;M'}$.  
\end{proof}

\begin{remark} \label{rmkforfinLinfty}
In fact, we have the following estimate for $v(x,t)$:
\begin{align*}
    \|v(\tau)\|_{H^{\ell}_{per}}
    &\leq e^{-\omega\tau}\|\phi\|_{H^{\ell}_{per}}+\frac{1}{\sqrt{2\Go}} \bigg(\int_0^{\tau}e^{-2\Go(\tau-s)}\|f(s)\|^2_{H^{\ell-2}_{per}}ds\bigg)^{\frac12}.
\end{align*}

If we further assume that $f\in L^{\infty}([0,\infty); H^{\ell-2}_{per})$ (in particular, if $f$ is periodic), then we have
    \begin{align*}
        \|v(\tau)\|_{H^{\ell}_{per}}\leq e^{-\omega\tau}\|\phi\|_{H^{\ell}_{per}}+\frac{1}{\sqrt{\Go}}\|f\|_{L^{\infty}([0,\infty);H^{\ell-2}_{per})},
    \end{align*}
    whence
    \begin{align}\label{estimateA}
        \|v\|_{Y^{\ell}_{\tau,T}}\leq C''_T (\|\phi\|_{H^{\ell}_{per}}+\|f\|_{L^{\infty}([0,\infty);H^{\ell-2}_{per})}),
    \end{align}
    where $C''_T$ is a constant independent of $\tau$.
\end{remark}

\begin{prop}
    Under the assumption of Proposition \ref{estw} and $c<1$, there exists $C>0$ such that the solution to \eqref{eqsub2}, $w$, satisfies
    $$\|w_t\|_{Y^{\ell}_{0,T}}\leq C  \bigg(\|\phi\|_{H^{\ell}_{per}}^2+\int_0^T\|f(\cdot,t)\|_{H^{\ell-2}_{per}}^2dt\bigg)^{\frac12}. $$
\end{prop}
\begin{proof}
    Note that 
    $$w_t=(I-\Delta)^{-1}(-ww_x-(vw)_x-vv_x)+Aw.$$
    Since $\ell\geq 1$, we know that $fg\in H^{\ell-1}_{per}$ whenever $f\in H^{\ell}_{per}$ and $g\in H^{\ell-1}_{per}$.
    Thus, 
    $$\|w(t)w_x(t)\|_{H^{\ell-2}_{per}}\leq C\|w(t)w_x(t)\|_{H^{\ell-1}_{per}}\leq C'\|w(t)\|_{H^{\ell}_{per}}\|w_x(t)\|_{H^{\ell-1}_{per}}\leq C''\|w(t)\|_{H^{\ell}_{per}}^2.$$
    We also obtain similar estimates for the term $(vw)_x$ and $vv_x$. Moreover, 
    $$\|Aw(t)\|_{H^{\ell}_{per}}\leq 3\|w(t)\|_{H^{\ell}_{per}}.$$
    Therefore, we have 
    \begin{align*}
        \|w_t\|_{Y^{\ell}_{0,T}}&\leq C (\|w\|^2_{Y^{\ell}_{0,T}}+\|v\|_{Y^{\ell}_{0,T}}^2)+C'\|w\|_{Y^{\ell}_{0,T}}\\
        &\lesssim  \bigg(\|\phi\|_{H^{\ell}_{per}}^2+\int_0^T\|f(\cdot,t)\|_{H^{\ell-2}_{per}}^2dt\bigg)^{\frac12}
    \end{align*}
    with the aid of Lemma \ref{estu} and Proposition \ref{estw}.
\end{proof}

We next address the case where the initial time in \eqref{sub1} and \eqref{eqsub2} is non-zero. For $t\geq \tau$, 
$$v(t)=e^{A(t-\tau)}v(\tau)+\int_\tau^te^{A(t-s)}(I-\Delta)^{-1}f(\cdot,s)ds$$
and 
$$w(t)=e^{A(t-\tau)}w(\tau)-\int_{\tau}^t e^{A(t-s)}(I-\Delta)^{-1}[w(s)w_x(s)+(v(s)w(s))_x+v(s)v_x(s)]ds$$
\begin{lemma}
    Let $\ell\geq 1$. Suppose $\phi\in H^{\ell}_{per}$ $f\in L^{2}([0,T+\tau];\cH^{\ell-2}).$ We have 
    $$\|v\|_{Y^{\ell}_{\tau,T}}\leq C''(\|\phi\|_{ H^{\ell}_{per}}+\|f\|_{L^2([0,\tau+T];H^{\ell-2}_{per})}),$$
    where $C''$ is independent of $\tau$.
\end{lemma}
\begin{proof}
    Observe that $v(\tau)=e^{A\tau}\phi+\int_0^{\tau}e^{A(\tau-s)}(I-\Delta)^{-1}f(s)ds$. Then, Lemma \ref{estu} gives us
    $$\|v(\tau)\|_{H^{\ell}_{per}}\leq C (\|\phi\|_{H^{\ell}_{per}}+\|f\|_{L^2([0,\tau];H^{\ell-2}_{per})}).$$
    Therefore, from the expression of $v$, we have 
    $$\|v(t)\|_{H^{\ell}_{per}}^2\leq e^{-\omega(t-\tau)}\|v(\tau)\|^2_{H^{\ell}_{per}}+\int_\tau^t e^{-\omega(t-s)}\|f(s)\|_{H^{\ell-2}_{per}}^2ds,$$
whence    
    $$\|v\|^2_{Y^{\ell}_{\tau,T}}\leq C' 2(\|\phi\|_{H^{\ell}_{per}}^2+\|f\|_{L^2([0,\tau];H^{\ell}_{per}}^2)+\int_{\tau}^{\tau+T}\|f(s)\|_{H^{\ell-2}_{per}}^2ds,$$
    and we have the desired inequality.
\end{proof}

\begin{prop} \label{wYlT}
Let $\ell\geq 1$. If $(\|w(\tau)\|_{H^{\ell}_{per}}^2+\|v\|_{Y^{\ell}_{\tau,T}}^2)^{\frac12}$ is sufficiently small, then 
    $$\|w\|_{Y^{\ell}_{\tau,T}}<K(\|w(\tau)\|_{H^{\ell}_{per}}^2+\|v\|_{Y^{\ell}_{\tau,T}}^2)^{\frac12}$$
for some $K>1$.
\end{prop}
\begin{proof}
    Define $\ds \Gamma(q)(t):=e^{A(t-\tau)}w(\tau)-\int_{\tau}^{t}e^{A(t-s)}(I-\Delta)^{-1}(qq_x+(qv)_x+vv_x)(s)ds$ for $t\geq \tau$, where $w$ is the solution to \eqref{eqsub2}.\\

 Consider the space $Y^\ell_{\tau,T,M'}:=\{q\in Y^{\ell}_{\tau,T}:\|q\|_{Y^{\ell}_{\tau,T}}\leq M'\}$ with $M'=K(\|w(\tau)\|_{H^{\ell}_{per}}^2+\|v\|_{Y^{\ell}_{\tau,T}}^2)^{\frac12}$.\\

    We will show the boundedness of $\Gamma$ on $Y^{\ell}_{\tau,T,M'}$. Contraction can be shown similarly by following the same argument as in the proof of Proposition \ref{estw}. In particular,     
    \begin{align*}
       \sup_{t\in [\tau,T+\tau]} \|\Gamma(q)(t)\|_{H^{\ell}_{per}}^2&\leq \|w(\tau)\|_{H^{\ell}_{per}}^2+9\int_\tau^{T+\tau}(\|qq_x\|_{H^{\ell-2}_{per}}^2+\|(qv)_x\|^2_{H^{\ell-2}_{per}}+\|vv_x\|_{H^{\ell-2}_{per}}^2)ds\\
        &\leq \|w(\tau)\|_{H^{\ell}_{per}}^2+9(\frac{3}{2}\|q\|_{Y^{\ell}_{\tau,T}}^4+\frac{3}{2}\|v\|_{Y^{\ell}_{\tau,T}}^4)\\
        &\leq \frac{(M')^2}{K^2}+\frac{27}{2}(M')^4+\frac{27}{2K^2}(M')^4.
        \end{align*}

    On the other hand, by Young's convolution inequality,

       \begin{align*}
     &\int_{\tau}^{T+\tau}\|\Gamma(q)(t)\|_{H^{\ell}_{per}}^2dt\\&\leq \|w(\tau)\|_{H^{\ell}_{per}}^2+9\int_{\tau}^{T+\tau}\int_\tau^{t}e^{-\omega(t-s)}(\|qq_x(s)\|_{H^{\ell-2}_{per}}^2+\|(qv)_x(s)\|^2_{H^{\ell-2}_{per}}+\|vv_x(s)\|_{H^{\ell-2}_{per}}^2)dsdt\\
        &\leq \|w(\tau)\|_{H^{\ell}_{per}}^2+9C(\frac{3}{2}\|q\|_{Y^{\ell}_{\tau,T}}^4+\frac{3}{2}\|v\|_{Y^{\ell}_{\tau,T}}^4)\\
        &\leq \frac{(M')^2}{K^2}+\frac{27C}{2}(M')^4+\frac{27C}{2K^2}(M')^4.
    \end{align*}
    
We can firstly choose $K>1$ and then $M'<1$ so that  $ \frac{(M')^2}{K^2}+\frac{27}{2}(M')^4+\frac{27}{2K^2}(M')^4<(M')^2$ and $\frac{(M')^2}{K^2}+\frac{27C}{2}(M')^4+\frac{27C}{2K^2}(M')^4<(M')^2$, then $\Gamma$ is bounded on $Y^{\ell}_{\tau,T,M'}$.\\

Thus, for some $K>1$ and sufficiently small $M'$, there exists a unique $q\in Y^{\ell}_{\tau,T,M'}$ such that $q=\Gamma(q)$ and $$\|q\|_{Y^{\ell}_{\tau,T}}<K(\|w(\tau)\|_{H^{\ell}_{per}}^2+\|v\|_{Y^{\ell}_{\tau,T}}^2)^{\frac12},$$
 and $q=w$. 
\end{proof}

We also need an estimate of $(w, w_t)$ with an upper bound relying on $v$.

\begin{prop}\label{Prop4}
Let $\ell\geq 1$ and assume that $\sup_{t\geq 0}\|v\|_{Y^{\ell}_{t,T}}^2$ is sufficiently small. Then, 
$$\|w\|_{Y^{\ell}_{\tau,T}}+\|w_t\|_{Y^{\ell}_{\tau,T}}\leq C'\sup_{t\geq 0}\|v\|_{Y^{\ell}_{t,T}}^2.$$
\end{prop}
\begin{proof}
    Consider the equation \eqref{eqsub2} on $[\tau, t+\tau]$ ($\tau,t>0$) is 
    \begin{align*}
        w(t+\tau)&=e^{At}w(\tau)-\int_{\tau}^{t+\tau}e^{A(t+\tau-s)}(I-\Delta)^{-1}[ww_x+(wv)_x](s)ds-\int_{\tau}^{t+\tau}e^{A(t+\tau-s)}(I-\Delta)^{-1}(vv_x)(s)ds \\
        &=:e^{At}w(\tau)+Q(t+\tau,\tau,w(\tau),v)+P(t+\tau,\tau,v).
    \end{align*}
    
    By Proposition \ref{wYlT}, it suffices to estimate $\|w(\tau)\|_{H^{\ell}_{per}}$ provided that $\|w(\tau)\|_{H^{\ell}_{per}}^2+\|v\|_{Y^{\ell}_{\tau,T}}^2$ is small.\\
    
    Define $w_k:=w(kT)$ for $k\in \N$. Then, applying the above relationship to $t=T$ and $\tau=(k-1)T$, we have 
    \begin{align*}
        w_k=e^{AT}w_{k-1}+Q(kT,(k-1)T,w_{k-1},v)+P(kT,(k-1)T,v)
    \end{align*}
    with $w_0=w(0)\equiv 0$.\\

    We first estimate $(w,w_t)$ when $\tau=(k-1)T$ for some $k\in \N$, and then for the case $(k-1)T<\tau<kT$ for some $k\in \N$. By Lemma \ref{bilinear}, 
    \begin{align*}
        &\|w_k\|_{H^{\ell}_{per}}\\
        &\leq e^{-\omega T}\|w_{k-1}\|_{H^{\ell}_{per}}+\int_{(k-1)T}^{kT}e^{-\omega (kT-s)}\big[\|ww_x\|_{H^{\ell-2}_{per}}+\|(wv)_x\|_{H^{\ell-2}_{per}}+\|vv_x\|_{H^{\ell-2}_{per}}\big]ds\\
        &\leq e^{-\omega T}\|w_{k-1}\|_{H^{\ell}_{per}}+\frac{C}{\sqrt{2\Go}}\Big(\|w\|_{Y^{\ell}_{(k-1)T,T}}^2+\|w\|_{Y^{\ell}_{(k-1)T,T}}\|v\|_{Y^{\ell}_{(k-1)T,T}}+\|v\|_{Y^{\ell}_{(k-1)T,T}}^2\Big) \\
        &\leq e^{-\omega T}\|w_{k-1}\|_{H^{\ell}_{per}}+\frac{C'}{\sqrt{2\Go}}\|w\|_{Y^{\ell}_{(k-1)T,T}}^2+\frac{C'}{\sqrt{2\Go}}\|v\|_{Y^{\ell}_{(k-1)T,T}}^2 \\
        &\leq  e^{-\omega T}\|w_{k-1}\|_{H^{\ell}_{per}}+\frac{C''K^2}{\sqrt{2\Go}}(\|w_{k-1}\|_{H^{\ell}_{per}}^2+\|v\|_{Y^{\ell}_{(k-1)T,T}}^2)+\frac{C'}{\sqrt{2\Go}}\|v\|_{Y^{\ell}_{(k-1)T,T}}^2 \\
        &\leq \bigg(e^{-\omega T}+\frac{C''K^2}{\sqrt{2\Go}}\|w_{k-1}\|_{H^{\ell}_{per}}\bigg)\|w_{k-1}\|_{H^{\ell}_{per}}+\frac{C''K^2+C'}{\sqrt{2\Go}}\|v\|_{Y^{\ell}_{(k-1)T,T}}^2
    \end{align*}    
    
If $\|w_{j}\|_{H^{\ell}_{per}}\leq \eta$ for $j\in \{1, \cdots,k-1\}$ and $\xi:=e^{-\omega T}+\frac{C''K^2}{\sqrt{2\Go}}\eta$, then we can rewrite the above inequality as 
$$\|w_{k}\|_{H^{\ell}_{per}}\leq \xi \|w_{k-1}\|_{H^{\ell}_{per}}+\bigg(\frac{C''K^2+C'}{2\sqrt{2\Go}}\bigg)\|v\|_{Y^{\ell}_{(k-1)T,T}}^2$$
Then, using iteration as in the proof of \cite{WZ}*{Lemma 3.4}, we have 
$$\|w_k\|_{H^{\ell}_{per}}\leq\bigg(\frac{C''K^2+C'}{2\sqrt{2\Go}}\bigg)\frac{1}{1-\xi}\sup_{t\geq 0}\|v\|_{Y^{\ell}_{t,T}}^2(\leq \eta).$$
provided that $\eta$ and $\sup_{t\geq 0}\|v\|_{Y^{\ell}_{t,T}}^2$ are sufficiently small so that $\xi<1$ and small enough to apply Proposition \ref{wYlT}.\\

For $\tau\in((k-1)T,kT)$, we write $\tau=(k-1)T+\Ga T$, where $\Ga\in(0,1)$, and
$$w(\tau)=e^{A(\Ga T)}w((k-1)T)-\int_{(k-1)T}^{(k-1)T+\Ga T}e^{A(\tau-s)}(I-\Delta)^{-1}[ww_x+(wv)_x+vv_x]ds.$$

A similar reasoning gives us 
\begin{align*}
    \|w(\tau)\|_{H^{\ell}_{per}}&\leq e^{-\omega \Ga T}\|w_{k-1}\|_{H^{\ell}_{per}}+\frac{3}{2\sqrt{2\Go}}\|w\|_{Y^{\ell}_{(k-1)T,\Ga T}}^2+\frac{3}{2\sqrt{2\Go}}\|v\|_{Y^{\ell}_{(k-1)T,\Ga T}}^2\\
    &\leq e^{-\omega \Ga T}\bigg(\frac{C''K^2+C'}{2\sqrt{2\Go}}\bigg)\frac{1}{1-\xi}\|v\|_{Y^{\ell}_{(k-2)T,T}}^2+\frac{3}{2\sqrt{2\Go}}\|w\|_{Y^{\ell}_{(k-1)T,\Ga T}}^2+\frac{3}{2\sqrt{2\Go}}\|v\|_{Y^{\ell}_{(k-1)T,\Ga T}}^2 \\
    &\leq e^{-\omega \Ga T}\bigg(\frac{C''K^2+C'}{2\sqrt{2\Go}}\bigg)\frac{1}{1-\xi}\|v\|_{Y^{\ell}_{(k-2)T,T}}^2+\frac{3}{2\sqrt{2\Go}}\|w\|_{Y^{\ell}_{(k-1)T,T}}^2+\frac{3}{2\sqrt{2\Go}}\|v\|_{Y^{\ell}_{(k-1)T,T}}^2 \\
    &\leq e^{-\omega \Ga T}\bigg(\frac{C''K^2+C'}{2\sqrt{2\Go}}\bigg)\frac{1}{1-\xi}\|v\|_{Y^{\ell}_{(k-2)T,T}}^2+\frac{3}{2\sqrt{2\Go}}(\|w_{k-1}\|^2_{H^{\ell}}+\|v\|_{Y^{\ell}_{(k-1)T,T}}^2)+\frac{3}{2\sqrt{2\Go}}\|v\|_{Y^{\ell}_{(k-1)T,T}}^2 \\
        &\leq C_{c,K,\xi}\sup_{t\geq 0}\|v\|_{Y^{\ell}_{t,T}}^2.
\end{align*}

 Combined with \eqref{estw}, it is inferred that 
 
\begin{align*}
    \|w_t\|_{Y^{\ell}_{\tau,T}}^2\leq C''\sup_{t\geq 0}\|v\|_{Y^{\ell}_{t,T}}^8+16C'\sup_{t\geq 0}\|v\|_{Y^{\ell}_{t,T}}^4.
\end{align*}

Since we can always let $\|v\|_{Y^{\ell}_{t,T}}$ sufficiently small for $t$: $\sup_{t\geq 0}\|v\|_{Y^{\ell}_{t,T}}<1$, we have \begin{align*}
    \|w_t\|_{Y^{\ell}_{\tau,T}}\leq C'''\sup_{t\geq 0}\|v\|_{Y^{\ell}_{t,T}}^2.
\end{align*}
\end{proof}

Owing to the above lemmas and propositions, we already paved all necessary estimates on linear and nonlinear parts of the solution, leading to the proof of Theorem \ref{thm1}. 

\begin{proof}[Proof of Theorem \ref{thm1}]\ \\
    Since $u(t)=v(t)+w(t)$, by Proposition \ref{wYlT} we obtain
    \begin{align*}
        \|u\|_{Y^{\ell}_{\tau,T}}&\leq \|v\|_{Y^{\ell}_{\tau,T}}+\|w\|_{Y^{\ell}_{\tau,T}} \\&\leq C_{K,c} (\|\phi\|_{H^{\ell}_{per}}+\|f\|_{L^2([0,\tau+T];H^{\ell-2}_{per})})\leq C_{K,c} (\|\phi\|_{H^{\ell}_{per}}+\|f\|_{L^2([0,\infty);H^{\ell-2}_{per})})
    \end{align*}
    provided that $ \|\phi\|_{H^{\ell}_{per}}+\|f\|_{L^2([0,\infty);H^{\ell-2}_{per})}$ is small so that $\sup_{t\geq 0}\|v\|_{Y^{\ell}_{t,T}}$ is small enough to apply Proposition \ref{Prop4}.  \\

    In the case of $f\in L^{\infty}([0,\infty);H^{\ell-2}_{per})$, by Remark \ref{rmkforfinLinfty}, it suffices to take $\|\phi\|_{H^{\ell}_{per}}+\|f\|_{L^\infty([0,\infty);H^{\ell-2}_{per})}$, but the smallness of this constant is depending on $T$ and independent of $\tau$.
\end{proof}

\section{Existence and stability of time-periodic solution with $\ell\ge 1$}\label{Sec4}

In this section, we assume that $f$ has time-periodicity: $f(x,t+\theta)=f(x,t)$ for all $t\geq 0$, i.e., $f(\cdot,t)$ has period $\theta$. Note that if $w(x,t):=u(x,t+\Gth) -u(x,t)$ we can see that $w$ solves 
\begin{equation}\label{eqmainperi}
\begin{cases}
    w_t-w_{xxt}+w_x+\frac{1}{2}([u(x,t+\Gth)+u(x,t)]w)_x+\Gb w-\Gg w_{xx}=0,  \quad (x, t)\in \bT\times [0, \infty),&\\
    w(x,0)=u(x,\theta)-\phi(x), &\\
    \lan w(t)\ran=0 \quad \forall t\geq 0.&
    \end{cases}
\end{equation}
Therefore, it suffices to consider
\begin{equation}\label{eqmainperi1}
\begin{cases}
    w_t-w_{xxt}+w_x+(wa)_x+\Gb w-\Gg w_{xx}=0,  \quad (x, t)\in \bT\times [0, \infty),&\\
    w(x,0)=\psi(x),&\\
    \lan w(t)\ran =0, \quad \forall t\geq 0&
    \end{cases}
\end{equation}
for some function $a$, and then we can apply the results to $a(x,t)=\frac{u(x,t+\Gth)+u(x,t)}{2}$. Since \eqref{eqmainperi1} is a linear PDE, the solution can be written as 
$$w(t)=e^{At}\psi(x)-\int_0^te^{A(t-s)}(I-\Delta)^{-1}(aw)_x(s)ds.$$

\begin{lemma}\label{aprieqap1}
    Let $\ell\geq 1$. Suppose $\|a\|_{Y^{\ell}_{0,T}}$ is sufficiently small, then 
    \begin{align}\label{w_l1}
        \|w\|_{Y^{\ell}_{0,T}}\leq C\|\psi\|_{H^{\ell}_{per}}
    \end{align}
    for some $C$ independent of $T$ and $\|a\|_{Y^{\ell}_{0,T}}$.\\

    If $\sup_{t\geq t_0}\|a\|_{Y^{\ell}_{t,T}}$ is small enough, then for any $\tau\geq t_0$
   \begin{align} \label{WYtauTtoWtau}
       \|w\|_{Y^{\ell}_{\tau,T}}\leq C\|w(\tau)\|_{H^{\ell}_{per}}
   \end{align}
   and
      \begin{align} \label{WtYtauTtoWtau}
       \|w_t\|_{Y^{\ell}_{\tau,T}}\leq C'\|w(\tau)\|_{H^{\ell}_{per}}.
   \end{align}
\end{lemma}
\begin{proof}
    From the semigroup solution of $w$, we have
    \begin{align*}
        \|w(t)\|_{H^{\ell}_{per}}^2&\leq 4e^{-\omega t}\|\psi\|_{H^{\ell}_{per}}^2+4\bigg(\int_0^te^{-\omega (t-s)}\|(aw)_x(s)\|_{H^{\ell-2}_{per}}ds\bigg)^2
    \end{align*}
    Then, following the argument as in the proof of Proposition \ref{estw}, one has that 
    $$\|w\|_{Y^{\ell}_{0,T}}^2\leq C\|\psi\|_{H^{\ell}_{per}}^2+C_{\Go}\|a\|_{Y^{\ell}_{0,T}}^2\|w\|_{Y^{\ell}_{0,T}}^2.$$
    Therefore, if $\|a\|_{Y^{\ell}_{0,T}}^2\leq \frac{1}{2C_{\Go}}$, then 
    $$\|w\|_{Y^{\ell}_{0,T}}^2\leq C'\|\psi\|_{H^{\ell}_{per}}^2,$$
    where $C'$ is independent of $a$.\\

     To estimate $\|w\|_{Y^{\ell}_{\tau,T}}$, we first note that
    $$w(t):=e^{A(t-\tau)}w(\tau)-\int_{\tau}^{t}e^{A(t-s)}(I-\Delta)^{-1}(aw)_x(s)ds.$$
    Then, 
    \begin{align*}
        \|w(t)\|_{H^{\ell}_{per}}\leq e^{-\omega (t-\tau)}\|w(\tau)\|_{H^{\ell}_{per}}+\int_{\tau}^{t}e^{-\omega (t-s)}\|(aw)_x(s)\|_{H^{\ell-2}_{per}}ds.
    \end{align*}
    Taking the supremum over $t\in [\tau,T+\tau]$, we have 
    \begin{align*}
       \sup_{t\in [\tau,T+\tau]} \|w(t)\|_{H^{\ell}_{per}}&\leq \|w(\tau)\|_{H^{\ell}_{per}}+\sup_{t\in [\tau, T+\tau]}\int_{\tau}^{t}e^{-\Go (t-s)}\|(aw)_x(s)\|_{H^{\ell-2}_{per}}ds\\
       &\leq \|w(\tau)\|_{H^{\ell}_{per}}+C''\|a\|_{Y^{\ell}_{\tau,T}}\|w\|_{Y^{\ell}_{\tau,T}}.
    \end{align*}
    
    If  we take $L^2([\tau,T+\tau])$ norm both sides, we have
     \begin{align*}
      \|w\|_{L^2([\tau,T+\tau];H^{\ell}_{per})}&\leq  \|e^{-\omega (t-\tau)}\|_{L^2([\tau,T+\tau])}\|w(\tau)\|_{H^{\ell}_{per}}+\bigg\|\int_{\tau}^{t}e^{-\omega (t-s)}\|(aw)_x(s)\|_{H^{\ell-2}_{per}}ds\bigg\|_{L^2([\tau,\tau+T])}\\
      &\leq \frac{1}{\sqrt{2\Go}}\|w(\tau)\|_{H^{\ell}_{per}}+C'''\|a\|_{Y^{\ell}_{\tau,T}}\|w\|_{Y^{\ell}_{\tau,T}}.
    \end{align*}
    Therefore, 
    $$\|w\|_{Y^{\ell}_{\tau,T}}\leq C^{(4)}\big(\|w(\tau)\|_{H^{\ell}_{per}}^2+\|a\|_{Y^{\ell}_{\tau,T}}^2\|w\|^2_{Y^{\ell}_{\tau,T}}\big)^{\frac{1}{2}}.$$
    If $\ds \sup_{t\geq t_0}\|a\|_{Y^{\ell}_{t,T}}$ is small enough, it can be inferred that  $$\|w\|_{Y^{\ell}_{\tau,T}}\leq C^{(5)}\|w(\tau)\|_{H^{\ell}_{per}}$$
holds for all $\tau\geq t_0$.\\

To estimate $\|w_t\|_{Y^{\ell}_{\tau,T}}$, since 
$$w_t=(I-\Delta)^{-1}[(aw)_x]+Aw,$$ 
we have
$$\|w_t(t)\|_{H^{\ell}_{per}}\leq \|(aw)_x(t)\|_{H^{\ell-2}_{per}}+C\|w(t)\|_{H^{\ell}_{per}}.$$
Now we take $L^{2}([\tau,T])$ and $L^{\infty}([\tau,T])$, and we have 
$$\|w_t\|_{Y^{\ell}_{\tau,T}}\leq C'\|a\|_{Y^\ell_{\tau,T}}\|w\|_{Y^{\ell}_{\tau,T}}+C\|w\|_{Y^{\ell}_{\tau,T}}\leq C''\|w\|_{Y^{\ell}_{\tau,T}}\leq C'''\|w(\tau)\|_{H^{\ell}_{per}}$$
provided that $\sup_{t\geq t_0}\|a\|_{Y^{\ell}_{\tau,T}}$ is small enough.
\end{proof}

We are now ready to prove Theorems \ref{thm2}-\ref{thm4}.

\begin{proof}[Proof of Theorem \ref{thm2}]\ \\
The proof of this theorem is essentially the same as in \cite{LW}; however, for the periodic domain, we do not need to apply any interpolation theorems to obtain the intermediate values of $\ell$. \\

Choose $\|\phi\|_{H^{\ell}_{per}}^2+\|f\|_{L^2([0,T];H^{\ell-2}_{per})}^2$  or $\|\phi\|_{H^{\ell}_{per}}^2+\|f\|_{L^\infty([0,\infty);H^{\ell-2}_{per})}^2$  small enough so that $\sup_{t\geq 0}\|\frac{u(x,t+\theta)+u(x,t)}{2}\|_{Y^{\ell}_{t,T}}$ is small enough to apply Lemma \ref{aprieqap1} with the aid of Theorem \ref{thm1}. We also take $\psi(x)=u(x,\theta)-\phi(x)$.\\

Note that
$$w(t+\tau)=e^{At}w(\tau)-\int_{0}^{t}e^{A(t-s)}(I-\Delta)^{-1}(aw)_x(s+\tau)ds.$$

We define $w_k=w(kT)$ for $k\in \N$. If $\tau=T$, then 
$$w_k=e^{AT}w_{k-1}-\int_{(k-1)T}^{kT}e^{A(kT-s)}(I-\Delta)^{-1}(aw)_x(s)ds.$$

Then for $\ell \in [1,\infty)$, we have 
\begin{align*}
    \|w_k\|_{H^{\ell}_{per}}&\leq e^{-\omega T}\|w_{k-1}\|_{H^{\ell}_{per}}+\int_{(k-1)T}^{kT}e^{-\omega (kT-s)}\|(aw)_x(s)\|_{H^{\ell-2}_{per}}ds\\
    &\leq  e^{-\omega T}\|w_{k-1}\|_{H^{\ell}_{per}}+\frac{1}{\sqrt{2\Go}}\|a\|_{Y^{\ell}_{(k-1)T,T}}\|w\|_{Y^{\ell}_{(k-1)T,T}}\\
    &\leq   e^{-\omega T}\|w_{k-1}\|_{H^{\ell}_{per}}+\frac{C'}{\sqrt{2\Go}}\|a\|_{Y^{\ell}_{(k-1)T,T}}\|w_{k-1}\|_{H^{\ell}_{per}} \\
    &= \bigg(e^{-\omega T}+\frac{C'}{\sqrt{2\Go}}\sup_{t\geq 0}\|a\|_{Y^{\ell}_{t,T}}\bigg)\|w_{k-1}\|_{H^{\ell}_{per}},
\end{align*}

By taking $\sup_{t\geq 0}\|a\|_{Y^{\ell}_{t,T}}$ small enough, we may assume that $e^{-\omega T}+\frac{C'}{\sqrt{2\Go}}\sup_{t\geq 0}\|a\|_{Y^{\ell}_{t,T}}=:\mu<1$. Note that $\mu$ is independent of $\tau$. Therefore, 
\begin{align*}
    \|w_k\|_{H^{\ell}_{per}}\leq \mu\|w_{k-1}\|_{H^{\ell}_{per}}\leq \mu^2\|w_{k-2}\|_{H^{\ell}_{per}}\leq \cdots\leq \mu^{k-1}\|w_1\|_{H^{\ell}_{per}}\leq \mu^k\|\psi\|_{H^{\ell}_{per}}.
\end{align*}
We see that 
$$\|w(kT)\|_{H^{\ell}_{per}}\leq e^{-kT(\ln(1/\mu)/T)}\|\psi\|_{H^{\ell}_{per}}.$$

If $\tau=kT+t'$, where $t'\in(0,T)$ and $k\in \N$, since 
\begin{align*}
   \|w(\tau)\|_{H^{\ell}_{per}}\leq \|w\|_{Y^{\ell}_{kT,t'}} \leq C'\|w_{k}\|_{H^{\ell}_{per}}&\leq C' e^{-kT(\ln(1/\mu)/T)}\|\psi\|_{H^{\ell}_{per}}\\
   &\leq C'e^{t'(\ln(1/\mu)/T)}e^{-\tau(\ln(1/\mu)/T)}\|\psi\|_{H^{\ell}_{per}}\\
    &\leq C'\frac{1}{\mu}e^{-\tau(\ln(1/\mu)/T)}\|\psi\|_{H^{\ell}_{per}}.
\end{align*}
Note that in the last step we used the fact that $t'\in (0,T)$.
Therefore, for any $t\geq 0$, we have 
$$\|w(t)\|_{H^{\ell}_{per}}\leq C'\frac{1}{\mu}e^{-t(\ln(1/\mu)/T)}\|\psi\|_{H^{\ell}_{per}},$$
and 
taking $L^2([\tau,T+\tau])$ as well as $L^{\infty}([\tau,T+\tau])$ we have 
$$\|w\|_{Y^{\ell}_{\tau,T}}\leq C'\bigg(\frac{\ln(1/\mu)}{T}+1\bigg)\frac{1}{\mu}e^{-\tau(\ln(1/\mu)/T)}\|u(\theta)-\phi\|_{H^{\ell}_{per}}.$$
These two inequalities are also true for $w_t$ with a factor of 3 of the original implicit majorizing constant. The proof is complete with the aid of Theorem \ref{thm1} and the smallness of $\|\phi\|_{H^{\ell}_{per}}+\|f\|_{L^\infty([0,\infty);H^{\ell-2}_{per})}$. \\

\end{proof}

\begin{proof}[Proof of Theorem \ref{thm3}]\ \\
Let $u_k(x)=u(x,k\theta)$ for $k\in \N$. By Theorem \ref{thm1}, we see $u_k\in H^{\ell}_{per}$. We will first show that $\{u_k\}_k$ is Cauchy in $H^{\ell}_{per}$, then we will show that the solution with initial condition $\lim_k u_k$ will be periodic.\\

    Let $m,n$ be integers. Then,
    \begin{align*}
        \|u_{n+m}-u_n\|_{H^{\ell}_{per}}&\leq \sum_{i=0}^{m-1}\|u_{n+i+1}-u_{n+i}\|_{H^{\ell}_{per}}\\
        &\leq \sum_{i=0}^{m-1}\|w((n+i)\theta)\|_{H^{\ell}_{per}}\\
        &\leq \sum_{i=0}^{m-1}C_{\mu,T}e^{-(n+i)\theta(\ln(1/\mu)/T)}\|u(\theta)-\phi\|_{H^{\ell}_{per}}\\
        &\leq C_{\mu,T}\|u(\theta)-\phi\|_{H^{\ell}_{per}}\frac{e^{-n\theta(\ln(1/\mu)/T)}}{1-e^{-\theta(\ln(1/\mu)/T)}}\to 0
    \end{align*}
    as $n\to\infty$. Therefore, $\{ u_{k}\}_k$ is a Cauchy sequence in $H^{\ell}_{per}$ and we will denote $\widetilde{\phi}:=\lim_{n\to\infty}u_n$ in $H^{\ell}_{per}$. We can see that $\|\widetilde{\phi}\|_{H^{\ell}_{per}}\leq C(\|\phi\|_{H^{\ell}_{per}}+\|f\|_{L^\infty([0,\infty);H^{\ell-2}_{per})}).$\\

    Now suppose $\widetilde{u}(x,t)$ be the solution to \eqref{main1} with initial condition $u(x,0)=\widetilde{\phi}$. Then
    \begin{align*}
    \|\widetilde{u}(\cdot,\theta)-\widetilde{\phi}(\cdot)\|_{H^{\ell}_{per}} &=\|\widetilde{u}(\cdot,\theta)-\widetilde{u}(\cdot,0)\|_{H^{\ell}_{per}}\\
        &\leq \|\widetilde{u}(\cdot,\theta)-u_{n+1}\|_{H^{\ell}_{per}}+\|u_{n+1}-u_n\|_{H^{\ell}_{per}}+\|u_n-\widetilde{u}(\cdot,0)\|_{H^{\ell}_{per}}\\
        &\to 0
    \end{align*}
    by passing $n\to \infty$ by the fact that $\{u_k\}_k$ is Cauchy and $\widetilde{\phi}:=\lim_{n\to\infty}u_n$ in $H^{\ell}_{per}$. To see the first term is small, note that by the mild solution of $\widetilde{u}$ and $u$ with $a(s)=\frac{\widetilde{u}(s)+u(n\theta+s)}{2}$, 
    \begin{align*}
        \widetilde{u}(\cdot,\theta)-u_{n+1}&= e^{A\theta}(\widetilde{\phi}-u_n)-\int_{0}^{\theta}e^{A(\theta-s)}(I-\Delta)^{-1}(a(s)[\widetilde{u}(s)-u(n\theta+s)])_xds
    \end{align*}
    and 
    \begin{align*}
        \|\widetilde{u}(\cdot,\theta)-u_{n+1}\|_{H^{\ell}_{per}}&\leq  C'e^{-\omega \theta}\|\widetilde{\phi}-u_n\|_{H^{\ell}_{per}}
    \end{align*}
    provided that $\sup_{t\geq 0}\|u\|_{Y^{\ell}_{t,\Gth}}+\sup_{t\geq 0}\|\widetilde{u}\|_{Y^{\ell}_{t,\Gth}}$ small enough by Lemma \ref{aprieqap1} if $\ell\geq 1$. \\

    We now show the local stability of $\widetilde{u}$. Consider $w(x,t)=u(x,t)-\widetilde{u}(x,t)$ and $a(x,t)=\frac{1}{2}(u(x,t)+\widetilde{u}(x,t))$. Then we see that $w$ solves \eqref{eqmainperi1} with $\psi(x)=\phi(x)-\widetilde{\phi}(x)$. By Theorem \ref{thm1}, we can take both $\sup_{t\geq 0}\|u\|_{Y^{\ell}_{t,T}}+\sup_{t\geq 0}\|\widetilde{u}\|_{Y^{\ell}_{t,T}}$ small enough to apply Theorem \ref{thm2}. Therefore, we have 
    $$\|u-\widetilde{u}\|_{Y^{\ell}_{\tau,T}}\leq C'_{\mu,T}e^{-\tau(\ln(1/\mu)/T)}\|\psi\|_{H^{\ell}_{per}},$$
    which is the exponential convergence.  
\end{proof}

\begin{proof}[Proof of Theorem \ref{thm4}]\ \\
To prove the global stability, we will first establish the case $\ell=1$, and then we will establish a method to reduce the higher-order cases $\ell$ to lower-order cases. Different from \cite{WZ}, differentiating $t$ would not give us higher-order regularity; we need to estimate the norm of $u$ in $H^{\ell}_{per}$ directly. The method of the proof is essentially the same as \cite{LW}; however, tracing the parameters $\Gb,\ \Gg$ is needed in our case. \\

In order to prove the global exponential stability for $H^{\ell}_{per}$, it suffices to show the global absorbing property

$$\|u(t)\|_{H^{\ell}_{per}}\leq e^{-ct}\|\phi\|_{H^{\ell}_{per}}+C\delta,$$
provided that $\sup_{t}\|f(t)\|_{H^{\ell-2}_{per}}\leq \delta$. Here we will always assume that $\delta <1$ for simplicity. We will explain how to obtain the global stability for higher $\ell$ after showing the global absorbing property for $\ell=1$.\\


We first consider $\ell=1$. If we multiply \eqref{main1} by $u$ and then integrate over $\bT$, then we have  
\begin{align*}
    \frac{1}{2}\frac{d}{dt}(\|u(t)\|_{L^2_{per}}^2+\|u_x(t)\|^2_{L^2_{per}})+\Gb \|u(t)\|_{L^2_{per}}^2+\Gg \|u_x(t)\|_{L^2_{per}}^2&\leq \frac{b}{4}(\|u(t)\|^2_{L^2_{per}}+\|u_x(t)\|^2_{L^2_{per}})+\frac{C}{b}\|f(t)\|_{H^{-1}_{per}}^2.
\end{align*}
By rearranging the terms and adding and subtracting $\Gb\|u_x(t)\|_{L^2_{per}}^2$ with the usage of Poincar\'e inequality, we have 
\begin{align*}
    &\frac{1}{2}\frac{d}{dt}(\|u(t)\|_{H^1_{per}}^2)+[\Gb-\frac{b}{4}+\frac{4\pi^2}{4\pi^2+1}(\Gg-\Gb)](\|u(t)\|_{H^1_{per}}^2)  \leq \frac{C}{b}\|f(t)\|_{H^{-1}_{per}}^2.
\end{align*}
Note that $\frac{4\pi^2}{4\pi^2+1}(\Gg-\Gb)+\Gb>0$ iff $\Gg>\frac{-\Gb}{4\pi^2}$, which we have assumed.

If $\Gg,\Gb>0$, we have the following inequality instead:
\begin{align*}
    &\frac{1}{2}\frac{d}{dt}(\|u(t)\|_{H^1_{per}}^2)+[\min\{\Gg,\Gb\}-\frac{b}{4}](\|u(t)\|_{H^1_{per}}^2)\leq \frac{C}{b}\|f(t)\|_{H^{-1}_{per}}^2.
\end{align*}
We now choose $\frac{b}{4}<\frac{4\pi^2}{4\pi^2+1}(\Gg-\Gb)+\Gb$ (or $\frac{b}{4}<\min\{\Gg,\Gb\}$ accordingly), and we can conclude that 
$$\|u(t)\|_{H^1_{per}}^2\leq C_be^{-c_b t}\|\phi\|_{H^1_{per}}^2+C_b'\sup_{t\geq 0}\|f(t)\|_{H^{-1}_{per}}^2.$$

Therefore, if $\delta>0$ is small such that there is a  $t_0>0$ that $\|u(t_0)\|_{H^1_{per}}^2$ is small enough so that $\ds \sup_{t\geq t_0}\|u\|_{Y^{1}_{t,T}}+\sup_{t\geq t_0}\|\widetilde{u}\|_{Y^{1}_{t,T}}$ is sufficiently small to apply Theorem \ref{thm3}. More precisely, for given $\phi \in H^1_{per}$ and sufficiently small $\ds\sup_{t\geq 0}\|f(t)\|_{H^{-1}_{per}}^2$, then there exists $t_0$ such that $\|u(t_0)\|_{H^1_{per}}$ is small, and we can apply Theorem \ref{thm3} to $u(t_0+t)$ and $u-\widetilde{u}\to 0$ in $Y^{\ell}_{\tau,T}$ as $\tau\to\infty$.\\


For $\ell>1$, noting that
\begin{align*}
(I-\GD)^{\frac{\ell}{2}}u_t+(I-\GD)^{\frac{\ell}{2}}u+(I-\GD)^{\frac{\ell-2}{2}}\partial_x[u^2]+(\Gb-\Gg) (I-\GD)^{\frac{\ell-2}{2}}u+\Gg (I-\GD)^{\frac{\ell}{2}}u=(I-\GD)^{\frac{\ell-2}{2}}f(t).
\end{align*}
and multiplying both sides by $(I-\Delta)^{\frac{\ell}{2}}u$, we have 
\begin{align*}
    &\frac{1}{2}\frac{d}{dt}\|u(t)\|_{H^{\ell}_{per}}^2+\|u(t)\|_{H^{\ell}_{per}}^2-\lan u^2(t),\partial_xu(t)\ran_{H^{\ell-1}_{per}}-(\Gb-\Gg)\|u(t)\|_{H^{\ell-1}_{per}}^2+\Gg\|u(t)\|_{H^{\ell}_{per}}^2\\
    &=\lan (I-\Delta)^{\frac{\ell-2}{2}}f(t), (I-\Delta)^{\frac{\ell}{2}}u(t)\ran_{L^2_{per}}.
\end{align*}
Then, by choosing $C_{\Gg}$ such that $AB\leq C_{\Gg}A^2+\Gg B^2$,
\begin{align*}
    \frac{1}{2}\frac{d}{dt}\|u(t)\|_{H^{\ell}_{per}}^2+\|u(t)\|_{H^{\ell}_{per}}^2 &\leq \|u^2(t)\|_{H^{\ell-1}_{per}}\|u(t)\|_{H^{\ell}_{per}}+(\Gg-\Gb) \|u(t)\|_{H^{\ell-1}_{per}}^2-\Gg\|u(t)\|_{H^{\ell}_{per}}^2+\|f(t)\|_{H^{\ell-2}_{per}}\|u(t)\|_{H^{\ell}_{per}} \\
     &\leq C_{\Gg}\|u^2(t)\|^2_{H^{\ell-1}_{per}}+(\Gg-\Gb)\|u(t)\|_{H^{\ell-1}_{per}}^2+C'_{\Gg}\|f(t)\|_{H^{\ell-2}_{per}}^2.
\end{align*}

For $1<\ell\leq 2$, if $\Gg-\Gb\geq 0$,
\begin{align} \label{13-1}
    &\frac{d}{dt}\|u(t)\|_{H^{\ell}_{per}}^2+2\|u(t)\|_{H^{\ell}_{per}}^2 \nonumber\\
     &\leq 2C_{\Gg}\|u(t)\|^4_{H^{1}_{per}}+2(\Gg-\Gb)\|u(t)\|_{H^{1}_{per}}^2+2C'_{\Gg}\|f(t)\|_{H^{\ell-2}_{per}}^2 \nonumber\\
     &\leq 2C_{\Gg}\big(C_be^{-c_bt}\|\phi\|_{H^1_{per}}^2+C_b'\sup_{s\geq 0}\|f(s)\|_{H^{-1}_{per}}^2\big)^2+2(\Gg-\Gb)\big(C_be^{-c_bt}\|\phi\|_{H^1_{per}}^2+C_b'\sup_{s\geq 0}\|f(s)\|_{H^{-1}_{per}}^2\big)+2C'_{\Gg}\|f(t)\|_{H^{\ell-2}_{per}}^2 \nonumber\\
     &\leq C_{\Gg,b}e^{-2c_bt}\|\phi\|_{H^1_{per}}^4+(\Gg-\Gb)C_{b}'e^{-c_bt}\|\phi\|_{H^1_{per}}^2+C_{b,\Gg,\Gb} \sup_{s\geq 0}\|f(s)\|_{H^{\ell-2}_{per}}^2,
\end{align}
in which we have used the fact that $\ell-1\leq 1<\ell$, $H^1_{per}$ forms an algebra, $\sup_{t\geq0 }\|f(t)\|_{H^{\ell-2}_{per}}<1$, and $H^{s+\Ge}_{per}\subset H^{s}_{per}$. If $\Gg-\Gb<0$, then one can drop the terms with coefficient $\Gg-\Gb$.\\

Now multiplying \eqref{13-1} both sides by $e^t$ and integrating over $[0,t]$, one can derive
\begin{align*}
   &e^t\|u(t)\|_{H^{\ell}_{per}}^2-\|\phi\|_{H^{\ell}_{per}}^2 \\
    &\leq \int_0^t \bigg[C_{\Gg,b}e^{(1-2c_b)s}\|\phi\|_{H^1_{per}}^4+(\Gg-\Gb)C_{b}'e^{(1-c_b)s}\|\phi\|_{H^1_{per}}^2+e^sC_{b,\Gg,\Gb} \sup_{T\geq 0}\|f(T)\|_{H^{\ell-2}_{per}}^2\bigg]ds\\
    &\leq C_{\Gg,b}'e^{t\max\{0,(1-2c_b)\}}\|\phi\|_{H^1_{per}}^4+(\Gg-\Gb)C_b''e^{t\max\{0,(1-c_b)\}}\|\phi\|_{H^1_{per}}^2+e^tC_{b,\Gg,\Gb} \sup_{T\geq 0}\|f(T)\|_{H^{\ell-2}_{per}}^2.
\end{align*}
Thus, 
$$\|u(t)\|_{H^{\ell}_{per}}^2\leq e^{-t}\|\phi\|_{H^{\ell}_{per}}^2+ C_{\Gg,b}'e^{t\max\{-1,-2c_b\}}\|\phi\|_{H^{\ell}_{per}}^4+(\Gg-\Gb)C_b''e^{t\max\{-1,-c_b\}}\|\phi\|_{H^{\ell}_{per}}^2+C_{b,\Gg,\Gb} \sup_{T\geq 0}\|f(T)\|_{H^{\ell-2}_{per}}^2.$$
We can conclude the absorbing property for $H^{\ell}_{per}$ because for given $\phi \in H^{\ell}_{per}$ and sufficiently small $\ds\sup_{t\geq 0}\|f(t)\|_{H^{\ell-2}_{per}}^2$, then there exists $t_0$ such that $\|u(t_0)\|_{H^{\ell}_{per}}$ is small, and we can apply Theorem \ref{thm3} to $u(t_0+t)$ to obtain $u-\widetilde{u}\to 0$ in $Y^{\ell}_{\tau,T}$ as $\tau\to\infty$.\\ 

For $\ell\geq 2$, note that $H^{\ell-1}_{per}$ always forms an algebra (which is not necessarily the case for $\ell\in (1,2]$), and we have $\|u^2\|_{H^{\ell-1}_{per}}\leq 
C\|u\|_{H^{\ell-1}_{per}}^2$. Thus, we have 
\begin{align*}
&\frac{d}{dt}\|u(t)\|_{H^{\ell}_{per}}^2+\|u(t)\|_{H^{\ell}_{per}}^2\leq C_{\Gg}\|u(t)\|^4_{H^{\ell-1}_{per}}+(\Gg-\Gb)\|u(t)\|_{H^{\ell-1}_{per}}^2+C'_{\Gg}\|f(t)\|_{H^{\ell-2}_{per}}^2.
\end{align*}

Therefore, we can obtain the desired estimate for $\ell\in (2,3]$ because we can apply the estimate for $\ell-1\in (1,2]$; and the result works inductively on $\ell\in (n,n+1]$. Therefore, we can conclude the absorbing property for $\ell\geq 2.$ Global stability results follow the same mechanism as that of $H^1_{per}$. \end{proof}


\section{The case of $0\le\ell<1$}
\label{Sec5}


In this section, we aim to discuss the aforementioned existence and stability results in low-regularity function spaces $H_{per}^\ell,\ \ell\in[0, 1)$. The analysis boils down to transforming the equation via the $I$ operator - $I_N$ - then applying the previous analysis to the equation about $I_Nu$. As we pointed out in this paper, there are almost conservative quantities such as the $H^1$-norm of $I_Nu$. Guided by this direction, the present authors have derived a similar calculation for the damped BBM equation posed on the whole line in \cite{LW2}. \\



At the beginning, we would define the $I$ operator. We define a smoothing bump function of frequency $k\in\mathbb{N}:$

\begin{eqnarray*}
{m_N}(k) = 
\begin{cases}
1, \,\,|k|\le N,\\\\
{|k|^{\ell-1}\over N^{\ell-1}}, \,\, |k|\ge 2N,
\end{cases}
\end{eqnarray*}
given an $N>0$. 


Linear operator $I_N$ so-called $I$-operator is that for fixed $N>0$: $H_{per}^\ell\mapsto H_{per}^1$ such that $\widehat{I_N g}(k) = m_N(k)\widehat g(k)$. There holds estimate between $\|\cdot\|_{H^{\ell}_{per}}$ and $\|I_N\cdot\|_{H^1_{per}}$:

\begin{lemma}\label{le5-2}
For $g\in H_{per}^\ell$ and a fixed $N>0$, 

$$C\|g\|_{H_{per}^\ell}\le\|I_Ng\|_{H_{per}^1}\le CN^{1-\ell}\|g\|_{H_{per}^1}.$$

Moreover for $u, v\in H_{per}^\ell$, there is constant $C_T$ relying on $T$ such that

\begin{eqnarray*}
\int_0^T\|(1-\partial^2_{x})^{-1}I_N(uv_x)(s)\|_{H_{per}^1}ds\le C_T\|I_Nu\|_{H_{per}^1}\|I_Nv\|_{H_{per}^1}.
\end{eqnarray*}
\end{lemma}
\begin{proof} It is a fact according to \cite{MWang2} that

$$\|(I-\partial^2_{x})^{-1}I_N\partial_x(u^2) \|\le C\|I_Nu\|^2_{H_{per}^1},$$
of which the derivation directly leads to the desired estimate.
\end{proof}

We apply $I_N$ on (\ref{main1}):

$$\partial_t I_Nu-\partial_t I_Nu_{xx}+I_Nu_x+\beta I_Nu -\gamma I_Nu_{xx}+I_N(uu_x) = I_Nf,$$
which can be split into two sub-problems:

\begin{equation} \label{lin7}
\begin{cases}
    \partial_t I_Nv-\partial_t I_Nv_{xx}+I_Nv_x+\beta I_Nv -\gamma I_Nv_{xx} = I_Nf, \quad\quad (x, t)\in \bT\times [0, \infty)&\\
    I_Nv(x,0)=I_N\phi(x), &
\end{cases}
\end{equation}
and 
\begin{equation}
\label{nonlin7}
\begin{cases}
\partial_t I_Nz-I_Nz_{xx}+I_Nz_x+\beta I_Nz -\gamma I_Nz_{xx}+I_N(zz_x+(zv)_x+vv_x) = 0, \quad\quad (x, t)\in \bT\times [0, \infty)&\\
I_Nz(x,0)=0.&
\end{cases}
\end{equation}

In the same fashion as the previous sections, there are estimates for \eqref{lin7} and \eqref{nonlin7}:\\

\begin{lemma}\label{le5-3}
For the linear problem (\ref{lin7}) on $\mathbb{T}$, for given $\tau\ge 0, T, N> 0, I_N\phi\in H_{per}^\ell$, and $I_N f\in L^2$, then $I_Nv\in Y^1_{\tau, T}$ and satisfies

\begin{eqnarray*}
\sup_{t\in{[\tau, \tau+T]}}\|I_Nv(t)\|_{H_{per}^1}\le \|I_N\phi\|_{H_{per}^1} + C\sup_{t\in{[\tau, \tau+T]}}\|I_Nf(t)\|, 
\end{eqnarray*} and 
\begin{eqnarray*}
\|I_Nv\|_{Y^1_{\tau, T}}\le C\|I_N\phi\|_{H_{per}^1} + C\sup_{t\in [\tau, \tau+T]}\|I_N f(t)\|. 
\end{eqnarray*}
\end{lemma}

\begin{lemma}\label{le5-4}
Problem (\ref{nonlin7}) admits a solution $I_Nz\in Y^1_{0, T}$ satisfying

\begin{eqnarray}
\|I_Nz\|_{Y^1_{0, T}}\le C\|I_Nv\|_{Y^1_{0, T}}.
\end{eqnarray} 

If $\tau>0$, on $[\tau, \tau+T]$,
\begin{eqnarray}\label{3-7}
\|I_Nz\|_{Y^1_{\tau, T}}\le 2\|I_Nz\|_{H_{per}^1} + 2\|I_Nv\|_{Y^1_{\tau, T}}.
\end{eqnarray}
\end{lemma}

\begin{lemma}\label{le5-5}
If $\tau>0$, $\|I_Nv\|_{Y^1_{0,T}}\le \delta'$, then there holds

\begin{eqnarray}
\|I_Nz\|_{Y^1_{\tau, T}}\le C\|I_N v\|_{Y^1_{0,T}}. 
\end{eqnarray}
\end{lemma}

 The stability property of the desired solution needs a key estimate: the decay estimate of the error equation. The solution to the error of the BBM equation is denoted by $w$. More precisely, on $\mathbb{T}$, we consider

\begin{eqnarray}\label{5-6}
\begin{cases}
w_t - w_{xxt} +\beta w- \gamma w_{xx} + (aw)_x = 0,\\
w(x, 0) = \psi(x),
\end{cases}
\end{eqnarray}
which has the form after applying $I_N$

\begin{eqnarray}\label{7-6}
\begin{cases}
\partial_tI_Nw - \partial_tI_Nw_{xx} +\beta I_Nw - \gamma I_N w_{xx} + I_N(aw)_x = 0,\\
I_Nw(x, 0) = I_N \psi(x),
\end{cases}
\end{eqnarray}

In the following content, we will specify the selection of $N$.  The function $a$ also needs some smallness as $v$. 

\begin{lemma}\label{le5-6}
If $\|I_Na\|_{Y^1_{\tau, T}}$ is bounded for $\tau\in \mathbb{R}_+$, there exists an integer $N_0$ so large that a constant $\gamma_0$ relying on $\beta, \gamma$, and  $N_0^{-{3\over 2}}\|I_Na\|_{H_{per}^1}$, then $\|I_Nw\|_{Y^1_{\tau, T}}\lesssim e^{-2\gamma_0\tau}$.
\end{lemma}
\begin{proof}
Multiplying $I_N w$ onto the equation and integrating on $\mathbb{T}$, we obtain 

\begin{eqnarray*}
{1\over 2}{d\over dt}\|I_Nw(t)\|_{H_{per}^1}^2 + \beta\|I_Nw(t)\|^2+\gamma\|I_Nw_{x}(t)\|^2=-(I_N(aw)_x, I_Nw).
\end{eqnarray*}

Noting that 
$$\|I_Nw_x\|^2 = \sum_{|k|\ge 1}m_N^2(k)|2\pi ik|^2{\hat w}^2(k),$$
the Poincaré-type inequality
$$\|I_Nw\|^2\le {1\over 4\pi^2}\|I_Nw_x\|^2,$$
and another straightforward calculation gives the observation

$$\|(I_Nw)_x\|^2=\|\sum_{|k|\ge1}m_N(k)(2\pi ik) \hat w(k)e^{2\pi ikx}\|^2=\|I_Nw_x\|^2,$$
and hence, $$\|I_Nw_x\|^2\le\|I_Nw\|_{H^1_{per}}^2\le {1+4\pi^2\over 4\pi^2}\|I_Nw_x\|^2.$$
If $\gamma >\max\{0, -{\beta\over 4\pi^2}\}$, then there exists
\begin{eqnarray*}
\Gg' = 
\begin{cases}
{4\pi^2\gamma\over 1+4\pi^2}, \,\,\beta\ge 0,\\\\
\gamma - {|\beta|\over 4\pi^2}, \,\, \beta<0,
\end{cases}
\end{eqnarray*}
such that $\Gg'>0$ and

\begin{eqnarray*}
{1\over 2}{d\over dt}\|I_Nw(t)\|_{H_{per}^1}^2 + \gamma'\|I_Nw(t)\|_{H^1_{per}}^2\le|(I_N(aw)_x, I_Nw)|.
\end{eqnarray*}

There holds a fact on torus (see e.g. \cite{MWang2}, and \cite{LW2} while the one in \cite{LW2} is for whole line case) that 

\begin{eqnarray}\label{5-7}
|(I_N(aw)_x, I_Nw)|\le CN^{-{3\over 2}}\|I_Na\|_{H_{per}^1}\|I_Nw\|^2_{H_{per}^1}.
\end{eqnarray} If we select $N=N_0$ so large that $\Gg'-CN_0^{-{3\over 2}}\|I_Na\|_{H_{per}^1}>0$, we are able to find a $\gamma_0$ such that

$${1\over 2}{d\over dt}\|I_Nw(t)\|^2_{H_{per}^1}+\gamma_0\|I_Nw(t)\|_{H_{per}^1}^2\le 0,$$
whence the exponential decay of $\|I_Nw\|_{H_{per}^1}$ follows in the sprit of ODE:

\begin{equation}\label{5-9}
\|I_Nw(t)\|_{H_{per}^1}^2\le\|I_N\psi\|_{H_{per}^1}^2e^{-2\gamma_0t}.
\end{equation}

To get decay estimate of $\int_{\tau}^{\tau+T}\|I_Nw(s)\|^2ds$ follows from integration result of above inequality: 
$$2\gamma_0\int_{\tau}^{\tau+T}\|I_Nw(s)\|_{H_{per}^1}^2ds\le \|I_Nw(\tau)\|_{H_{per}^1}^2+\|I_Nw(\tau+T)\|_{H_{per}^1}^2.$$

Based on the above decay estimates, the estimate of $\|w\|_{H^\ell_{\tau, T}}$ can be obtained via Lemma \ref{le5-2}. 
\end{proof}

We now proceed to prove the main results for $0 \le \ell < 1$. \\

\begin{proof} [Proof of Theorem \ref{thm1} for $0\le\ell<1$] \ \\
For fixed $\tau, T>0$, owing to Lemmas \ref{le5-3} - \ref{le5-5}, we obtain the sum $I_Nv, I_Nz\in Y^{1}_{\tau, T}$ then $u = v+z \in Y^\ell_{\tau, T}$ via Lemma \ref{le5-2}. 
\end{proof}

\begin{proof}[Proof of Theorem \ref{thm2} for $0\le\ell<1$]\ \\
Noting that $f$ is $\theta$-time-periodic, if $w = u(x, t+\theta)-u(x,t)$, then $w$ exactly follows (\ref{5-6}) with $\psi(x) = u(x, \theta)-\phi(x), a={1\over 2}(u(x,t)+u(x, t+\theta))$ and $I_Nw$ follows (\ref{7-6}). If $\delta$ is sufficiently small, $\|I_Na\|_{H^1}$ is such small quantity that $\gamma_0>0$ of (\ref{7-6}) in Lemma \ref{le5-6}. Thus the asymptotic periodicity of $u$ is obtained as Lemma \ref{le5-6} and Lemma \ref{le5-2}:

$$\|u(\cdot, \cdot+\theta)-u(\cdot, \cdot)\|_{Y^{\ell}_{\tau, T}}\lesssim e^{-\rho\tau}$$

with $\rho=2\gamma_0.$
\end{proof}

\begin{proof}[Proof of Theorem \ref{thm3} for $0\le \ell<1$]\ \\
The sketch is that, noting $u_k=u(x, k\theta), k\in \mathbb{N}$, we could show that the sequence $\{u_k\}_k$ is Cauchy in $H^{\ell}$, whose limit, $\displaystyle\lim_{k\rightarrow\infty}u_k = \tilde\phi$, could be picked as a new initial value of the original PDE. Its solution trajectory $\tilde u$ is precisely time-periodic and in $H^{\ell}$. It is apparent due to Theorem \ref{thm1} that since the size in norm of $\tilde\phi$ is small, the size in norm of $\tilde u$ is small as well. \\

Secondly, we deal with the uniqueness of $\tilde u$. If the uniqueness doesn't hold otherwise, there are at least two distinct periodic solutions, $\tilde u_1$ and $\tilde u_2$. Define $w=\tilde u_1(x,t) - \tilde u_2(x,t)$, which satisfies (\ref{5-9}) again. Since $\psi\equiv0$ in this case, we have $w\equiv 0$, which proves the uniqueness of $\tilde u.$\\

Thirdly, we address the local stability of the solution. Let $w = u(x,t)-\tilde u(x,t),\ a={u(t,x)+\tilde u(t, x)\over 2}$. Then, $w$ satisfies (\ref{5-6}). From Lemma \ref{le5-3}-\ref{le5-5}, $I_Na$ is small in size as $I_Nu$ is small on $[\tau, \tau+T]$. Lemma \ref{le5-6} giving the exponential decay leads to the local stability as per Definition \ref{de2-2}.
\end{proof}

\begin{proof}[Proof of Theorem \ref{thm4} for $0\le\ell<1$]\ \\
Similar to the analysis in \cite{LW2}, it suffices to prove the global absorbing property of $I_Nu$. Note that

\begin{eqnarray*}
\partial_t I_Nu - \partial_tI_Nu_{xx} + \beta I_Nu  - \gamma I_Nu_{xx} + I_N(uu_x)= I_Nf,
\end{eqnarray*}
multiplying $I_Nu$ onto it, then similar calculation as that of $I_Nw$ in Lemma \ref{le5-6} gives

\begin{eqnarray*}
{1\over 2}{d\over dt}\|I_Nu(t)\|_{H^1_{per}}^2 + \Gg'\|I_Nu(t)\|_{H^1_{per}}^2 = (I_N(f-uu_x), I_Nu),
\end{eqnarray*} with same $\gamma'$ specified in Lemma \ref{le5-6}.\\



There hold

\begin{eqnarray*}
|(I_Nf, I_Nu)|&\le& {1\over 2}\|f\|^2+{1\over 2}\|I_Nu\|^2
\end{eqnarray*} by Young's inequality, and 

\begin{eqnarray*}
|(I_Nuu_x, I_Nu)|&\le& CN^{-{3\over 2}+}\|I_Nu\|^3_{H^1_{per}},
\end{eqnarray*}
by a similar calculation of (\ref{5-7}).\\

Pick $N_1\in\mathbb{N}$ so large that $\gamma_1 = \Gg' -{1\over 8\pi^2} - CN_1^{-{3\over 2}+}\|I_Nu\|_{H^1_{per}}>0$, then it can be inferred that

\begin{eqnarray*}
{d\over dt}\|I_Nu(t)\|_{H^1_{per}}^2+2\gamma_1\|I_Nu(t)\|_{H^1_{per}}^2 \le \|f(t)\|^2.
\end{eqnarray*}

Applying an integration factor in the fashion of ODE, we have

\begin{eqnarray*}
\|I_Nu(t)\|_{H^1_{per}}^2&\le& e^{-{2\gamma_1 t}}\|I_N\phi\|^2_{H^1_{per}}+\int_0^te^{-2\gamma_1{(t-s)}}\|f(s)\|^2ds\\
&\le& e^{-{2\gamma_1 t}}\|I_N\phi\|^2_{H^1_{per}}+{C(\gamma_1)}\sup_{t\in[0, \infty]}\|f(t)\|^2. 
\end{eqnarray*}

Like the mechanism shown in \cite{LW2}, the local stability combined with this absorbing behavior gives global stability of $I_Nu$ defined in Definition \ref{de2-3}. Global stability of $u$ in $H^\ell$ agrees with this behavior by Lemma \ref{le5-2}.
\end{proof}

\begin{remark}[Generalized BBM]\label{multi}\ \\
    We end this paper by giving a remark addressing a generalized BBM equation of the form
\begin{equation} \label{mainp}
\begin{cases}
    u_t-u_{xxt}+u_x+u^{p-1}u_x+\Gb u-\Gg u_{xx}=f(x,t), \quad\quad (x, t)\in \bT\times [0, \infty),&\\
    u(x,0)=\phi(x).&
\end{cases}
\end{equation}

    In a recent paper \cite{KimK}, Kim and Kwak have proved that the generalized BBM equation (with $\Gb=0=\Gg$ and $f\equiv0$) is locally well-posed on $C([0,T];H^\ell_{per})$ for $\ell\geq \frac{p-2}{2p}$ and globally well-posed for $(p,\ell)\in \{3\}\times [\frac{1}{4},\infty)$ and $(p,\ell)\in \{5\}\times (\frac{1}{2},\infty)$. Moreover, they showed that the result of local well-posedness is sharp by showing that the multilinear estimate fails to hold on $H^\ell_{per}$ for $\ell<\frac{p-2}{2p}$. More precisely, they proved the following multilinear estimate with sharp exponent $\ell$:
    
    \begin{lemma}[\cite{KimK}] \label{kimk}
        Suppose $p$ is an integer with $p\geq 2$. Then for any $\ell\geq \frac{p-2}{2p}$, there exists $C>0$ such that 
        $$\|(I-\Delta)^{-1}\partial_x(\prod_{j=1}^p f_j)\|_{H^\ell_{per}}\leq C\prod_{j=1}^p \|f_j\|_{H^\ell_{per}}$$
        for all $\{f_j\}_{j=1}^{p}\subset H^\ell_{per}$.
    \end{lemma}

    \begin{lemma} \label{multi2}
        For $\ell'\in [\frac{p-2}{2p},\ell]$ and $p\in \N\cap [2,\infty)$, we have 
        \begin{align*}
            \int_{s}^{s+t} \bigg\|\partial_x\bigg(\prod_{j=1}^p f_j(s)\bigg)\bigg\|_{H^{\ell'-2}_{per}}^2ds &\leq C \prod_{j=1}^p \|f_j\|_{Y^\ell_{s,t}}^2.
        \end{align*}
    \end{lemma}
    \begin{proof}
    It suffices to consider the case $s=0$, and the proof for general $s$ would be similar. Note that
        \begin{align*}
            \int_0^t\bigg\|\partial_x\bigg(\prod_{j=1}^p f_j(s)\bigg)\bigg\|_{H^{\ell'-2}_{per}}^2ds&=\int_0^t\bigg\|(I-\Delta)^{-1}\partial_x\bigg(\prod_{j=1}^p f_j(s)\bigg)\bigg\|_{H^{\ell'}_{per}}^2ds \\&\leq C\int_0^t \prod_{j=1}^{p}\|f_j(s)\|_{H^{\ell'}_{per}}^2 ds \\
            &\leq \frac{C}{p}\sum_{j=1}^p\bigg(\int_0^t\|f_j(s)\|_{H^{\ell'}_{per}}^2ds \bigg)\prod_{k\neq j}\sup_{0\leq s\leq t}\|f_k(s)\|_{H^{\ell'}_{per}}^2 \\
            &\leq \frac{C}{p}\sum_{j=1}^p\bigg(\int_0^t\|f_j(s)\|_{H^{\ell}_{per}}^2ds \bigg)\prod_{k\neq j}\sup_{0\leq s\leq t}\|f_k(s)\|_{H^{\ell}_{per}}^2 \\
            &\leq \frac{C}{p}\sum_{j=1}^{p}\|f_j\|_{Y^{\ell}_{0,t}}^2\prod_{k\neq j}\|f_k\|_{Y^{\ell}_{0,t}}^2\\&= C\prod_{j=1}^{p}\|f_j\|_{Y^\ell_{0,t}}^2.
        \end{align*}
    \end{proof}
    In generalized BBM case, if $a(x,t)=\frac{1}{p}\sum_{j=0}^{p-1} [u(x,t+\Gth)]^{j}[u(x,t)]^{p-1-j}$, then $w(x,t)=u(x,t+\Gth)-u(x,t)$ is the solution to \eqref{eqmainperi1}. To obtain Theorems \ref{thm2} and \ref{thm3}, one needs $\sup_{t\geq 0}\|a\|_{Y^{\ell}_{t,T}}$ sufficiently small, which can be achieved by the fact that $H^\ell_{per}$ forms an algebra for $\ell\geq 1>\frac{1}{2}$. Indeed, by Lemma \ref{multi2}, one has 
    \begin{align*}
        \|f^{j}g^{p-1-j}\|_{Y^{\ell}_{t,T}}&=\sup_{s\in [t,T+t]}\|f^jg^{p-1-j}\|_{H^\ell_{per}}+\|f^jg^{p-1-j}\|_{L^2([t,T+t];H^\ell_{per})}\\
        &\leq \sup_{s\in [t,T+t]}\|f^jg^{p-1-j}\|_{H^\ell_{per}}+\bigg(\int_t^{t+T}\|f(s)\|_{H^\ell_{per}}^{2j}\|g(s)\|_{H^\ell_{per}}^{2(p-1-j)}ds\bigg)^{\frac{1}{2}} \\
        &\leq 2\|f\|_{Y^{s}_{t,T}}^j\|g\|_{Y^s_{t,T}}^{p-1-j}.
    \end{align*}
    Thus, as long as taking $\|\phi\|_{H^{\ell}_{per}}^2+\|f\|_{L^\infty([0,\infty);H^{\ell-2}_{per})}^2$ small enough so that $\sup_{t\geq 0}\|a\|_{Y^{\ell}_{t,T}}$ is sufficiently small to apply Lemma \ref{aprieqap1}. The remainder of the argument proceeds exactly as the proof of Theorems \ref{thm2} and \ref{thm3}.\\
    
    Also, to prove Theorem \ref{thm4}, noticing that $\lan u^{p-1}u_x\ran =0$, no changes are necessary to obtain the result other than the bilinear estimate. Therefore,  Theorems \ref{thm1}-\ref{thm4} hold for generalized BBM equation for $\ell\geq 1$.\\

    For $\ell <1$, following the proof of Theorems \ref{thm2} and \ref{thm3} for the case $\ell\geq 1$ by taking $a(x,t)=\sum_{j=0}^{p-1}[u(x,t+\Gth)]^{j}[u(x,t)]^{p-1-j}$, $a(x,t)=\sum_{j=0}^{p-1}[\widetilde{u}(x,t)]^{j}[u(x,n\Gth+t)]^{p-1-j}$, and $a(x,t)=\sum_{j=0}^{p-1}[u(x,t)]^{j}[\widetilde{u}(x,t)]^{p-1-j}$ in sequence, we can prove Theorems \ref{thm2} and \ref{thm3} for $\ell\in [\frac{p-2}{2p},1]$ with the aid of Lemma \ref{multi2}. To obtain Theorem \ref{thm4}, we consider the $I$-method again. From \cite{MWang2}*{Lemma 2.4}, one has 
    $$\|(I-\Delta)^{-1}\partial_xI_N(\prod_{j=1}^p f_j)\|_{H^1_{per}}=\|(I-\Delta)^{-1}I_N\partial_x(\prod_{j=1}^p f_j)\|_{H^1_{per}}\leq C\prod_{j=1}^p \|I_Nf_j\|_{H^1_{per}}$$
    provided that all $f_j\in H^{\ell}_{per}$ for given $\ell\in [\frac{1}{2}-\frac{1}{p},1)$. The first equality holds because $I_N$ is also a Fourier multiplier. Also,  one needs to apply \cite{MWang2}*{Lemma 2.6}, namely, 
    $$|(I_N\partial_x(u^p), I_N u)|\leq C N^{-\frac{3}{2}+}\|I_N(u)\|_{H^1_{per}}^{p+1}.$$ By choosing $N_1\in \N$ large enough so that 
    $$\Gg_1=\Gg' -\frac{1}{8\pi^2}-CN_1^{-\frac{3}{2}+}\|I_Nu\|_{H^1_{per}}^{p-1}>0,$$ we can conclude the theorem by following the proof of Theorem \ref{thm4} for $0\leq \ell<1$ ($p=2$).\\


    We now summarize the discussion above as the following theorem: 

    \begin{theorem}
        Let $p\in \N\cap[2,\infty)$, $\Gb\in \R$, $\Gg>\max\{0,-\frac{\beta}{4\pi^2}\}$, $\ell\in [\frac{1}{2}-\frac{1}{p},\infty)$, and $f$ be a $H^{\ell-2}_{per}$-valued $\Gth$-periodic with $\lan f(t)\ran=0$ for all $t\in [0,\Gth)$.
        \begin{enumerate}
            \item For $\tau,\, T>0$, there exists $\delta>0$ such that if  $$\|\phi\|_{H^{\ell}_{per}}+\|f\|_{L^{\infty}([0,\Gth);H^{\ell}_{per})}\leq \delta,$$ then there exists a unique $u\in L^{\infty}\cap L^2([\tau,T];H^\ell_{per,0})$ such that $u$ is a mild solution to \eqref{mainp} and  
    $$\|u\|_{Y^{\ell}_{\tau,T}}\leq C(\|\phi\|_{H^{\ell}_{per}}+\|f\|_{L^{\infty}([0,\Gth);H^{\ell-2}_{per})}). $$ 
    \item For any $T,\,\tau>0$ and the solution $u$ in (1), there exists $C,\,\rho>0$ such that 
    $$\|u(\cdot, \cdot+\theta)-u(\cdot, \cdot)\|_{Y^{\ell}_{\tau,T}}\leq C e^{-\rho \tau}.$$
    \item Equation \eqref{mainp} has a locally stable and unique periodic solution in $H^{\ell}_{per}$.
    \item There exists $\Gd'>0$ such that \eqref{mainp} has an essentially globally stable periodic solution provided that $\ds \sup_{t\geq 0}\|f(t)\|_{H^{\ell-2}_{per}}\leq \Gd'$.
        \end{enumerate}
    \end{theorem}
\end{remark}

\section*{Declaration of competing interests}
The authors declare that they have no known competing financial interests or personal relationships that could have appeared to influence the work reported in this paper.

\section*{Acknowledgments}
The authors sincerely thank the reviewer for their valuable comments and diligent work. \\

\noindent One of the authors, Taige Wang, is supported by Faculty Development Fund granted by College of Arts and Sciences, University of Cincinnati, and Taft Award by Taft Research Center, University of Cincinnati. The authors would like to take this chance to thank them for their constant support.


\begin{thebibliography}{lllp}
\bibitem{Amick} C. Amick, J. Bona, and M. Schonbek, Decay of solutions of some nonlinear wave equations, J. Diff. Equ., 81 (1989), 1–49. 

\bibitem{BBM} T. Benjamin, J. Bona, and J. Mahony, Model equations for long waves in nonlinear dispersive systems, Philos. Trans. Royal Soc. London Ser. A, 272 (1972), 47 - 78.





\bibitem{BSZ2} J. Bona, S. Sun, and B.-Y. Zhang, Forced oscillations of a damped Korteweg-de Vries equation in a quarter plane, Comm. Cont. Math., 5 (3) (2003), 369-400.




\bibitem{Tzvet} J. Bona and N. Tzvetkov, Sharp well-posedness results for the BBM equation, Disc. Cont. Dyn. Syst., 23 (4) (2009), 1241 - 1252. 

\bibitem{BW} J. Bona and J. Wu, Temporal growth and eventual periodicity for dispersive wave equations in a quarter plane, Disc. Cont. Dyn. Syst., 23 (4) (2009), 1141 - 1168.  


\bibitem{Bourgain} J. Bourgain, Fourier transform restriction phenomena for certain lattice subsets and applications to nonlinear evolution equations II. The KdV-equation, Geom. Funct. Anal., 3 (1993), 209 - 262. 

\bibitem{Tao1} J. Colliander, M. Keel, G. Staffilani, H. Takaoka, and T. Tao, Sharp global well-posedness for KdV and modified KdV on $\mathbb{R}$ and $\mathbb{T}$, J. Amer. Math. Soc., 16 (3) (2003), 705 - 749. 

\bibitem{Tao2} J. Colliander, M. Keel, G. Staffilani, H. Takaoka, and T. Tao, Multilinear estimates for periodic KdV equations and applications, J. Func. Anal., 211 (2004), 173 - 218. 

\bibitem{HChen} H. Chen, Periodic initial-value problem for BBM-equation, Comp. Appl. Math., 48 (2004),  1305–1318. 


\bibitem{Kenig} C.E. Kenig, G. Ponce, and L. Vega. A bilinear estimate with applications to the KdV equation. J. Amer. Math. Soc., 9 (1996), 573–603. 

\bibitem{KimK} S. Kim and C. Kwak, Well-posedness issues for the generalized Benjamin--Bona--Mahony equation, arXiv:2603.21060, preprint.

 \bibitem{LW} C. Lau and T. Wang, On periodic solutions of the Benjamin-Bona-Mahony-Burgers equation, Appl. Anal., to appear.

\bibitem{LW2} C. Lau and T. Wang, Forced oscillation of a damped BBM equation posed on whole line in low regularity spaces, submitted, available at arxiv: 2602.08327. 


\bibitem{LZ} Z. Liu and S. Zheng, Semigroups Associated with Dissipative Systems, Vol. 398., CRC Press, 1999.




\bibitem{Russell} D. Russell and B.-Y. Zhang, Smoothing and decay properties of solutions of the Korteweg-de Vries equation on a periodic domain with point dissipation, J. Math. Anal. Appl., 190 (1995), 449 -- 488.

\bibitem{RZ} D. Russell and B.-Y. Zhang, Exact controllability and stabilizability of the Korteweg-de Vries equation, Trans. Amer. Math. Soc., 348 (9) (1996), 3643-3672. 

\bibitem{UsmanZ2} M. Usman and B.-Y. Zhang, Forced oscillations of the Korteweg-de Vries equation on a bounded domain and their stability, Disc. Cont. Dyn. Sys., 26 (4) (2010), 1509 - 1523.

\bibitem{MWang2} M. Wang, Long time behavior of a damped generalized BBM equation in low regularity spaces, Math. Meth. Appl. Sci., 38 (2015), 4852 - 4866. 

\bibitem{WZ} T. Wang and B.-Y. Zhang, Forced oscillation of viscous Burgers' equation with a time-periodic force, Disc. Cont. Dyn. Syst. B, 26 (2) (2021), 1205 - 1221.


\bibitem{Yang} X. Yang and B.-Y. Zhang, Local well-posedness of the coupled KdV-KdV systems on $\mathbb{R}$, Evol. Equ. Control Theory, 11(5) (2022), 1829-1871.

\bibitem{Yang2} X. Yang and B.-Y. Zhang, Well-posedness and critical index set of the Cauchy problem for the couple KdV-KdV systems on $\mathbb{T}$, Disc. Cont. Dyn. Syst., 42 (11) (2022), 5167 - 5199. 
\end{thebibliography}
\end{document}